\documentclass[11pt]{amsart}

\usepackage{graphics}
\usepackage{graphicx}
\usepackage[T1]{fontenc}    

\usepackage{amsthm}
\usepackage{amsbsy,amsmath,amssymb,amscd,amsfonts}
\usepackage{hyperref}
\usepackage{url}            
\usepackage{booktabs}       
\usepackage{nicefrac}       
\usepackage{microtype}      

\usepackage{float,latexsym,color}
\usepackage[font={scriptsize,it}]{caption}
\usepackage{subcaption}
\usepackage{tablefootnote}

\usepackage{makecell}

\usepackage[dvipsnames]{xcolor}

\newtheorem{theorem}{Theorem}
\newtheorem{step}{Step}

\newtheorem{proposition}{Proposition}
\newtheorem{remark}{Remark}

\newtheorem{lemma}{Lemma}
\newtheorem*{lemma*}{Lemma}

\newtheorem{affirmation}{Affirmation}

\def\cp#1{{\mathbb{P}}_{#1}\mathbb{C} }
\hypersetup{
    pdftoolbar=true,        
    pdfmenubar=true,        
    pdffitwindow=false,     
    pdfstartview={FitH},    
    colorlinks=true,       
    linkcolor=OliveGreen,          
    citecolor=blue,        
    filecolor=black,      
    urlcolor=red           
}

\usepackage{enumerate} 
\usepackage[pagewise]{lineno}
\usepackage[export]{adjustbox}
\DeclareMathOperator*{\argmin}{arg\,min} 

\title[3-periodics: why so many ellipses?]{Loci of 3-periodics in an Elliptic Billiard:\\Why so many ellipses?}

\author{Ronaldo Garcia}
\address{Ronaldo Garcia\\
Inst. de Matemática e Estatística\\
Univ. Federal de Goiás\\
Goiânia, GO, Brazil}
\email{ragarcia@ufg.br}

\author{Jair Koiller}
\address{Jair Koiller\\
Dept. de Matemática\\
Univ. Federal de Juiz de Fora\\
Juiz de Fora, MG, Brazil}
\email{jairkoiller@gmail.com}

\author{Dan Reznik}
\address{Dan Reznik\\
Data Science Consulting\\
Rio de Janeiro, RJ, Brazil}
\email{dan@dat-sci.com}

\begin{document}

\begin{abstract}
A triangle center such as the incenter, barycenter, etc., is specified by a function thrice- and cyclically applied on sidelengths and/or angles. Consider the 1d family of 3-periodics in the elliptic billiard, and the loci of its triangle centers. Some will sweep ellipses, and others higher-degree algebraic curves. We propose two rigorous methods to prove if the locus of a given center is an ellipse: one based on computer algebra, and another based on an algebro-geometric method. We also prove that if the triangle center function is rational on sidelengths, the locus is algebraic.\\



\end{abstract}

\modulolinenumbers[1]

\maketitle

\section{Introduction}
\label{sec:intro}
Classic notable points of a triangle include the incenter, barycenter, orthocenter, and circumcenter, see Appendix~\ref{app:constr} for a refresher. While these are traditionally obtained via geometric constructions, Kimberling has generalized this to {\em triangle centers} \cite{kimberling1993_rocky}, specified by a {\em triangle function}, thrice-applied (in cyclical fashion) to the sidelengths and/or angles of a triangle, thus forming a triple of {\em trilinear coordinates}. These are reviewed in Appendix~\ref{app:trilin}.

Thousands of such centers are catalogued in Kimberling's Encyclopedia of Triangle Centers (ETC) \cite{etc}. Centers are labeled $X_i$, e.g., the incenter is $X_1$, the barycenter $X_2$, etc. For each center the corresponding triangle function is provided. A quick perusal reveals these to be either rational, irrational, or more rarely, transcendental, on the sidelengths and/or angles of a triangle.

Consider the 1d family of 3-periodic orbits in the elliptic billiard (EB), Figure~\ref{fig:3-periodics}. The vertices are bisected by ellipse normals and the sides are tangent to a virtual confocal caustic; see Appendix~\ref{app:billiards} for a review. We have been drawn to this family because unexpectedly, the locus of the incenter is an ellipse and that of the Mittenpunkt\footnote{This point was discovered by Nagel in 1836 as the point of concurrence of lines from the excenters through side midpoints \cite{mw}, see Figure~\ref{fig:constructions} (right).} $X_9$ is the billiard center \cite{reznik2020-intelligencer}. Additionally, many other curious invariants have been detected and/or proved \cite{akopyan2020-invariants,bialy2020-invariants,reznik2020-invariants}. 

\begin{figure}
    \centering
    \includegraphics[width=.5\textwidth]{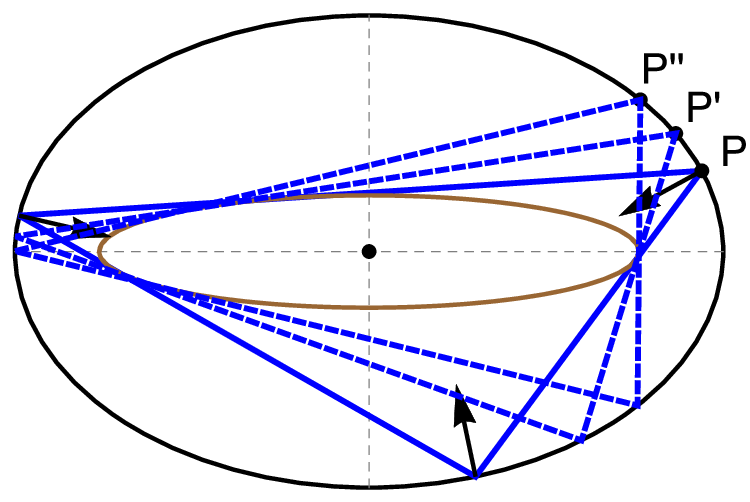}
    \caption{Three 3-periodic orbits (one solid with vertex $P$ and two dashed with vertices $P'$, $P''$). The vertices are bisected by ellipse normals and the sides are dynamically tangent to a virtual confocal caustic (brown). Remarkably, the family conserves perimeter \cite{sergei91}. \href{https://bit.ly/38oncCD}{app}}
    \label{fig:3-periodics}
\end{figure}

In general, triangle centers sweep such curves as ellipses, quartics, sextics, etc., with or without self-intersections, etc.; see Figure~\ref{fig:incenter-loci}. The central question here is: given a triangle center, is it possible to predict its locus curve type over billiard 3-periodics based on the triangle function?

\begin{figure}
    \centering
    \includegraphics[width=\textwidth]{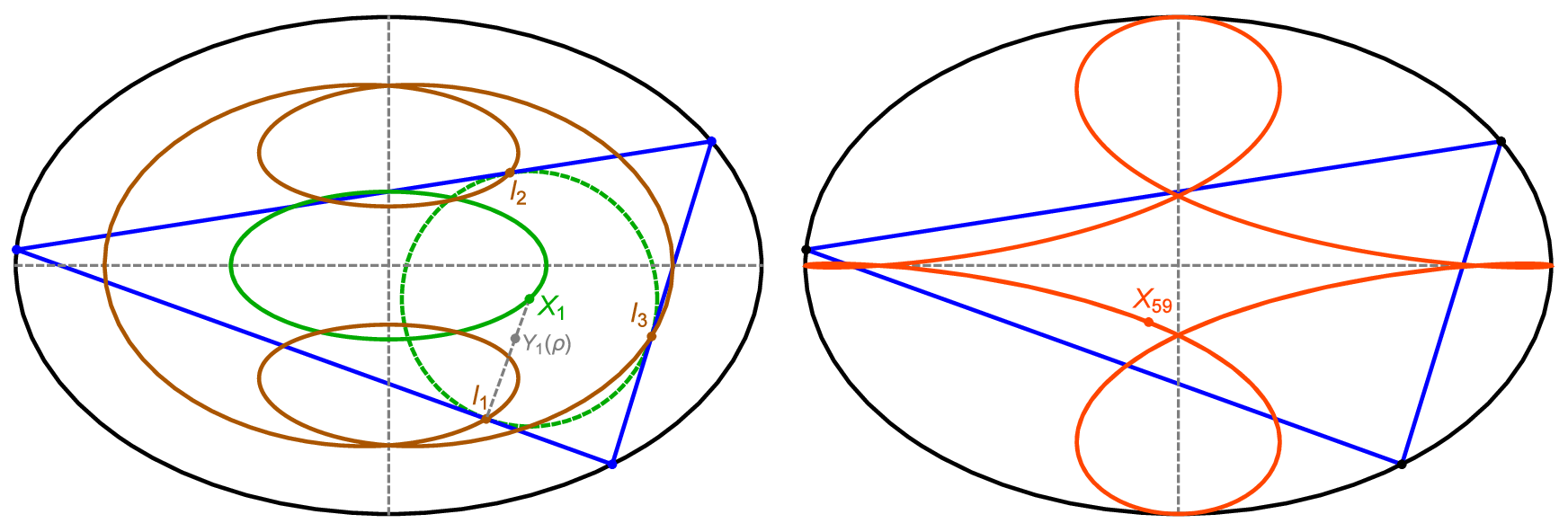}
    \caption{\textbf{Left}: A sample 3-periodic orbit (blue), the elliptic locus of the Incenter $X_1$ (green) and of that of the Intouchpoints $I_1,I_2,I_3$ (brown), the points of contact of the Incircle (dashed green) with the orbit sides. These produce a curve with two internal lobes whose degree is at least 6. $Y_1(\rho)$ is a convex combination of $X_1$ and $I_1$ referred to in Section~\ref{sec:triple-winding}. \href{https://bit.ly/3q4b0Nn}{app},  \href{https://youtu.be/BBsyM7RnswA}{Video 1}, \href{https://youtu.be/9xU6T7hQMzs}{Video 2}. \textbf{Right}: The locus of $X_{59}$ is a curve with four self-intersections. A vertical line epsilon away from the origin intersects the locus at 6 points. \href{https://bit.ly/3i4h6dX}{app}}
    \label{fig:incenter-loci}
\end{figure}

\subsection*{Main Results}

We present a hybrid numerical-CAS (Computer Algebra System) method to rigorously verify if the locus of given triangle center is an ellipse or not. We apply it to the first 100 centers listed in \cite{etc}, finding that 29 are elliptic. We have derived explicit expressions for their axes \cite{garcia2021-ellipses-web} (Theorem~\ref{main}).

We consider the curious case of the Symmedian Point $X_6$, whose locus closely approximates an ellipse but using CAS-based manipulation, we show it is actually a quartic (Theorem~\ref{thm:x6}), which we derive explicitly. Interestingly, its triangle function is rational on the sidelengths and is one of the simplest on the whole of \cite{etc}; see Table~\ref{tab:center-trilinears} in Appendix~\ref{app:constr}.

Using algebro-geometric techniques, we prove (Theorem~\ref{thm:rational-center}) that when the trilinear coordinates (defined in Appendix~\ref{app:trilin}) of a triangle center are rational on the orbit's sidelengths, the locus is an algebraic curve (elliptic or otherwise). We also describe an algorithm based on the method of resultants which computes the zero set of a polynomial in two variables corresponding to the locus. After extensive experimentation, we haven't yet found a non-rational triangle function which results in an elliptic locus, so it is likely that the latter requires a rational triangle function.

\subsection*{Related Work}

Odehnal extensively studied loci of triangle centers over the poristic triangle family \cite{odehnal2011-poristic}. The early experimental result that the locus of the Incenter $X_1$ of billiard 3-periodics is an ellipse was subsequently proven \cite{olga14,garcia2019-incenter}. Proofs soon followed for the ellipticity of both $X_2$ \cite{sergei2016-com} and $X_3$ \cite{corentin19,garcia2019-incenter}; see Figures \ref{fig:non-elliptic-vertex} and \ref{fig:x12345-feuer-combo} in Appendix~\ref{app:early}. More recently, a theory for the ellipticity of a locus based on certain properties of a triangle is being developed, see \cite{helman2021-theory}. For a textbook on Poncelet-based phenomena see \cite{garcia2021-impa}.

\subsection*{Outline}

Our main methods appear in Sections~\ref{sec:loci_geom} and \ref{sec:algebraic}. We conclude in Section~\ref{sec:conclusion} with a list of questions and interesting links and videos. The Appendices contain supporting material. Most figures contain clickable links to relevant videos and/or the experiment displayed on our browser-based \href{https://dan-reznik.github.io/ellipse-mounted-loci-p5js}{app} \cite{darlan2020-ellipse-mounted}.

\section{Detecting Elliptic Loci}
\label{sec:loci_geom}
Let the boundary of the EB be given by ($a>b>0$):
\begin{equation}
f(x,y)=\left(\frac{x}{a}\right)^2+\left(\frac{y}{b}\right)^2=1.
\label{eqn:billiard-f}
\end{equation}

We describe a numerically-assisted method\footnote{We started with visual inspection, but this is both laborious and unreliable, some loci (take $X_{30}$ and $X_6$ as examples) are indistinguishable from ellipses to the naked eye.} which proves that the locus of a given triangle center is elliptic or otherwise. We then apply it to the first 100 triangle centers. listed in \cite{etc}. Indeed, a much larger list could be tested.

\subsection{Proof Method}

Our proof method consists of two phases, one numeric, and one symbolic, Figure~\ref{fig:method-pipeline}. It makes use of the following Lemmas, whose proofs appear in Appendix~\ref{app:method-lemmas}:

\begin{lemma}
The locus of a triangle center $X_i$ is symmetric about both EB axes and centered on the latter's origin.
\label{lem:axisymmetric}
\end{lemma}

\begin{proof} 
Given a 3-periodic $T=P_1P_2P_3$. Since the EB is symmetric about its axes, the family will contain the reflection of $T$ about said axes, call these $T'$ and $T''$. Since triangle centers are invariant with respect to reflections, Appendix~\ref{app:triangle-centers}, $X_i'$ and $X_i''$ will be found at similary reflected locations.
\end{proof}

Note that  if the locus is elliptic, the above implies it will be concentric and axis-aligned with the EB.


%

\begin{lemma}
\label{lem:axis-of-symmetry}
Any triangle center $X_i$ of an isosceles triangle is on the axis of symmetry of said triangle.
\end{lemma}

\begin{lemma}
\label{lem:center-cover}
A parametric traversal of $P_1$ around the EB boundary triple covers the locus of triangle center $X_i$, elliptic or not.
\end{lemma}

See Section~\ref{sec:triple-winding} for more details.

A first phase fits a concentric, axis-aligned ellipse to a fine sampling of the locus of some triangle center $X_i$. A good fit occurs when the error is several orders of magnitude\footnote{Robust fitting of ellipses to a cloud of points is not new \cite{fitzgibbon99-ellipse}. In our case, the only source of error in triangle center coordinates is numerical precision, whose propagation can be bounded by Interval Analysis \cite{moore2009-interval-analysis,snyder92-ellipse}.} less than the sum of axes regressed by the process. False negatives are eliminated by setting the error threshold to numeric precision. False positives can be produced by adding arbitrarily small noise to samples of a perfect ellipse, though this type of misclassification does not survive the next, symbolic phase. 

A second phase attempts to symbolically verify via a Computer Algebra System (CAS) if the parametric locus of the $X_i$ satisfies the equation of a concentric, axis-aligned ellipse. Expressions for its semi-axes are obtained by evaluating $X_i$ at isosceles orbit configurations, Figure~\ref{fig:sideways-upright-orbit}. The method is explained in detail in Figure~\ref{fig:method-detail}.

\begin{figure}
    \centering
    \includegraphics[width=.5\textwidth]{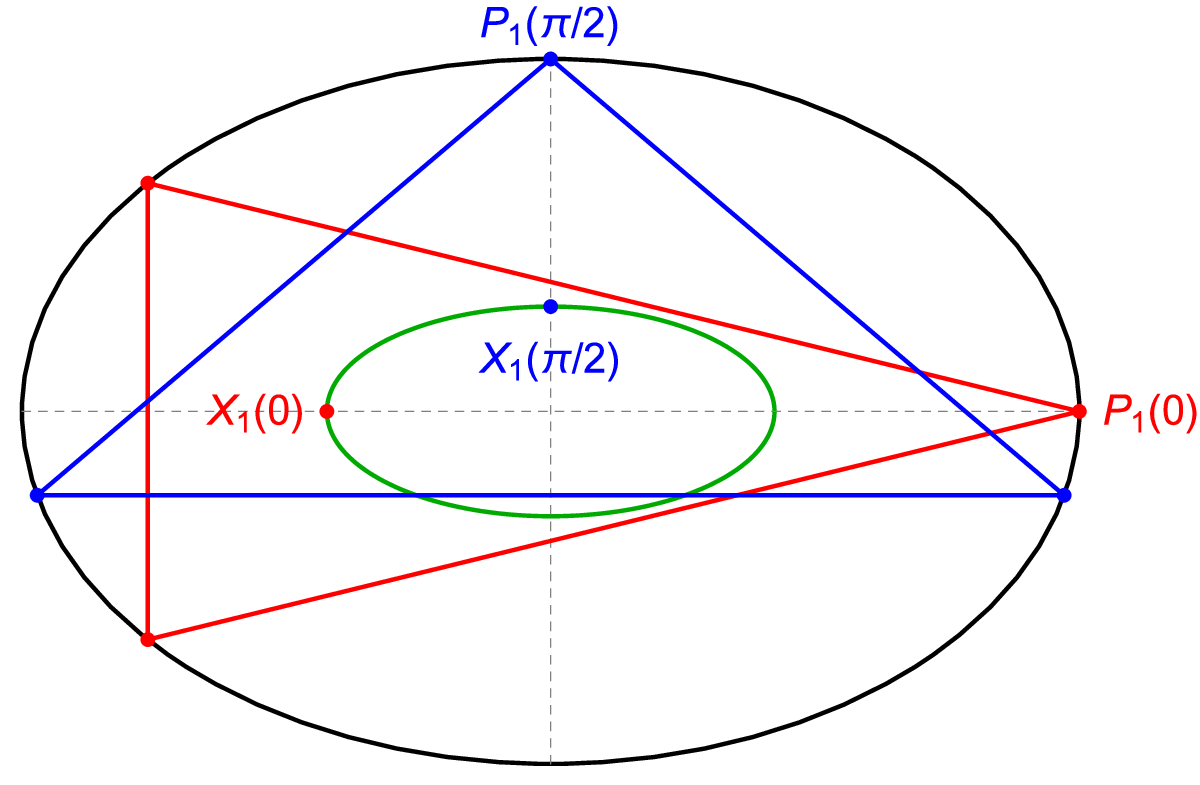}
    \caption{With $P_1$ at the right (resp.~top) vertex of the EB, the orbit is a sideways (resp.~upright) isosceles triangle, solid red (resp.~solid blue). Not shown are their two symmetric reflections. Also shown (green) is the locus of a sample triangle center, $X_1$ in this case. At the isosceles positions, vertices will lie on the axis of symmetry of the triangle, Lemma~\ref{lem:axis-of-symmetry}. When the locus is elliptic, the $x,y$ coordinates of $X_i(0),X_i(\pi/2)$ are the semi-axes $a_i,b_i$, respectively.}
    \label{fig:sideways-upright-orbit}
\end{figure}

\begin{figure}
    \centering
    \includegraphics[width=.9\textwidth]{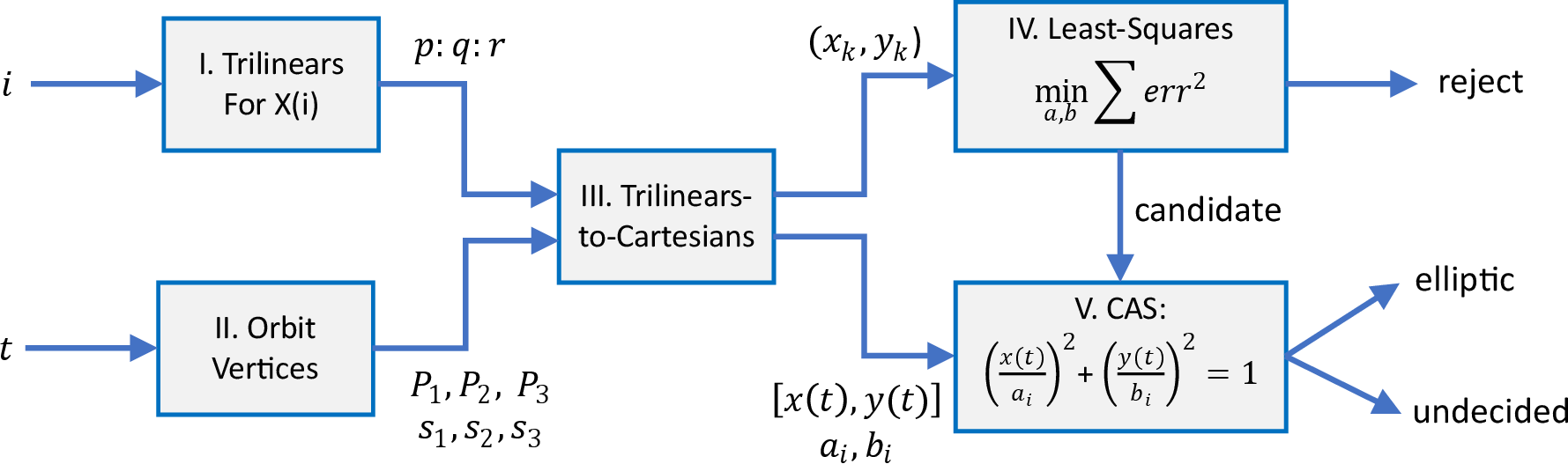}
    \caption{Our method as a flow-chart, modules labeled from I to V: (I) Once the ith center is specified,  its trilinears $p:q:r$ (see Appendix~\ref{app:triangle-centers}) are obtained from the Encyclopedia of triangle centers (ETC) \cite{etc}. (II) Given a symbolic parameter $t$ or a numeric sample $t_k$, obtain cartesian coordinates for orbit vertices and the sidelengths, Appendix~\ref{app:p1p2p3}. (III) Combine trilinears and orbit data to obtain, via \eqref{eqn:trilin-cartesian}, numeric cartesian coordinates for $X_i(t_k)$ or symbolically as $X_i(t)$. (IV) Least-squares fit an axis-aligned, concentric ellipse to the $(x_k,y_k)$ samples and accept $X_i$ as candidate if the fit error is sufficiently small. (V) Verify, via a Computer Algebra System (CAS), if the parametric locus $X_i(t)$ satisfies the equation of an ellipse whose axes $a_i,b_i$ are obtained symbolically in II-III by setting $t=0,\pi/2$. If the CAS is successful, $X_i(t)$ is deemed elliptic, otherwise the result is ``undecided''. The CAS was successful for all 29 candidates selected by IV out of the first 100 Kimberling centers.}
    \label{fig:method-pipeline}
\end{figure}






\begin{figure}
\fbox{
\begin{minipage}{\textwidth}
\footnotesize
\noindent \textbf{1. Select Candidates:}
\begin{itemize}
\item Let the EB have axes $a,b$ such that $a>b>0$. Calculate $P_1(t_k)=\left(a\cos(t_k),b\sin(t_k)\right)$, for $M$ equally-spaced samples $t_k\in[0,2\pi)$, $k=1,2,...,M$.
\item Obtain the cartesian coordinates for the orbit vertices $P_2(t_k)$ and $P_3(t_k)$, $\forall k$ (Appendix~\ref{app:p1p2p3}).
\item Obtain the cartesian coordinates for triangle center $X_i$ from its trilinears \eqref{eqn:trilin-cartesian}, for $\forall{t_k}$. If analyzing the vertex of a derived triangle, convert a row of its {\em trilinear matrix} to cartesian coordinates, Appendix~\ref{app:derived-tris}.
\item Least-squares fit an origin-centered, axis-aligned ellipse (2 parameters) to the $X_i(t_k)$ samples, Lemma~\ref{lem:axisymmetric}. Accept the locus as potentially elliptic if the numeric fit error is negligible, rejecting it otherwise.
\end{itemize}
\noindent \textbf{2. Verify with CAS}
\begin{itemize}
    \item Taking $a,b$ as symbolic variables, calculate $X_i(0)$ (resp.~$X(\pi/2)$), placing $P_1$ at the right (resp.~top) EB vertex. The orbit will be a sideways (resp.~upright) isosceles triangle, Figure~\ref{fig:sideways-upright-orbit}. By Lemma~\ref{lem:axis-of-symmetry}, $X_i$ will fall along the axis of symmetry of either isosceles.
    \item The $x$ coordinate of $X_i(0)$ (resp.~the $y$ of  $X_i(\pi/2)$) will be symbolic expressions in $a,b$. Use them as candidate locus semiaxes' lengths $a_i,b_i$.
    \item Taking $t$ as a symbolic variable, use a computer algebra system (CAS) to verify if $X_i(t)=\left(x_i(t),y_i(t)\right)$, as parametrics on $t$, satisfy $(x_i(t)/a_i)^2+(y_i(t)/b_i)^2=1$, ${\forall}t$.
    \item If the CAS is successful, Lemma~\ref{lem:center-cover} guarantees $X_i(t)$ is will cover the entire ellipse, so assert that the locus of $X_i$ is an ellipse. Else, locus ellipticity is indeterminate. 
\end{itemize}
\end{minipage}}
\caption{Method for detecting ellipticity of a triangle center locus.}
\label{fig:method-detail}
\end{figure}

\subsection{Phase 1: Least-Squares-Based Candidate Selection}

Let the position of $X_i(t)$ be sampled (randomly or uniformly) at $t_k\in[0,2\pi]$, $k=1{\ldots}M$. If the locus is an ellipse, than the latter is concentric and axis-aligned with the EB, Lemma~\ref{lem:axisymmetric}. Express the squared error as the sum of squared sample deviations from an implicit ellipse:
\begin{equation*}
\text{err}^2(a_i,b_i)= \sum_{k=1}^M{\left[\left(\frac{x_k}{a_i}\right)^2+\left(\frac{y_k}{b_i}\right)^2-1\right]^2}
\end{equation*}

\noindent Least-squares can be used to estimate the semi-axes:

$$
(\hat{a_i},\hat{b_i})=\argmin_{a,b}\left\{ \text{err}^2(a_i,b_i)\right\}
$$

The first 100 Kimberling centers separate into two distinct clusters: 29 with negligible least-squares error, and 71 with finite ones. These are shown in ascending order of error in our companion website \cite[Part II]{garcia2021-ellipses-web}.

A gallery of loci generated by $X_1$ to $X_{100}$ (as well as vertices of several derived triangles) is provided in \cite{dsr_locus_gallery_2019}.

\subsection{Phase 2: Symbolic Verification with a CAS}

A CAS was successful in symbolically verifying that all 29 candidates selected in Phase 1 satisfy the equation of an ellipse (none were undecided). As an intermediate step, explicit expressions for their elliptic semi-axes were computed and appear in \cite[Part I]{garcia2021-ellipses-web}.

\begin{theorem} \label{main}
Out of the first 100 centers in \cite{etc}, exactly 29 produce elliptic loci, all of which are concentric and axis-aligned with the EB. These are $X_i$,i=1, 2, 3, 4, 5, 7, 8, 10, 11, 12, 20, 21, 35, 36, 40, 46, 55, 57, 63, 65, 72, 78, 79, 80, 84, 88, 90. Specifically:
\begin{itemize}
\item The loci of $X_i,i=2,7,57,63$ are ellipses homothetic to the EB.
\item The loci of $X_i,i=4,10,40$ are ellipses homothetic to a $90^\circ$-rotated copy of the EB.
\item The loci of $X_i,i=88,100$ are ellipses identical to the EB.
\item The loci of $X_{55}$ is an ellipse homothetic to the $N=3$ caustic.
\item The loci of $X_i,i=3,84$ are ellipses homothetic to a $90^\circ$-rotated copy of the $N=3$ caustic.
\item The locus of $X_{11}$ is an ellipse identical to the $N=3$ caustic.
\end{itemize}
\end{theorem}

The above above are summarized on Table~\ref{tab:ell}. The least-square fit errors for the first 100 Kimberling are shown in Figure~\ref{fig:least-squares-error-graph}. 

\begin{remark}
The loci of $X_i$, $i=1,2,3,4$ are the ellipses $x^2/a_i^2+y^2/b_i^2=1$. Via CAS, their semi-axes $(a_i,b_i)$ are given by:
{\small  
\begin{align*}
    a_1&= \frac{\delta-b^2}{a} & b_1&= \frac{a^2-\delta }{b} &\\
    a_2&=k_2 a & b_2&= k_2 b \\
    a_3&=\frac{a^2-\delta}{2b} & b_3&= \frac{\delta-b^2}{2a} & \\
    a_4&=\frac{k_4}{a} & b_4&= \frac{k_4}{b}
\end{align*}
where $\delta^2=a^4+b^4-a^2b^2$, $k_2=(2\delta-a^2-b^2)/c^2$, $k_4=[(a^2+b^2)\delta-2a^2b^2)]/c^2$, and $c^2=a^2-b^2$.
}
\end{remark}

Explicit expressions for the locus semi-axes for abovementioned centers appear in \cite[Part I]{garcia2021-ellipses-web}. 

\begin{figure}
    \centering
    \includegraphics[width=.7\textwidth]{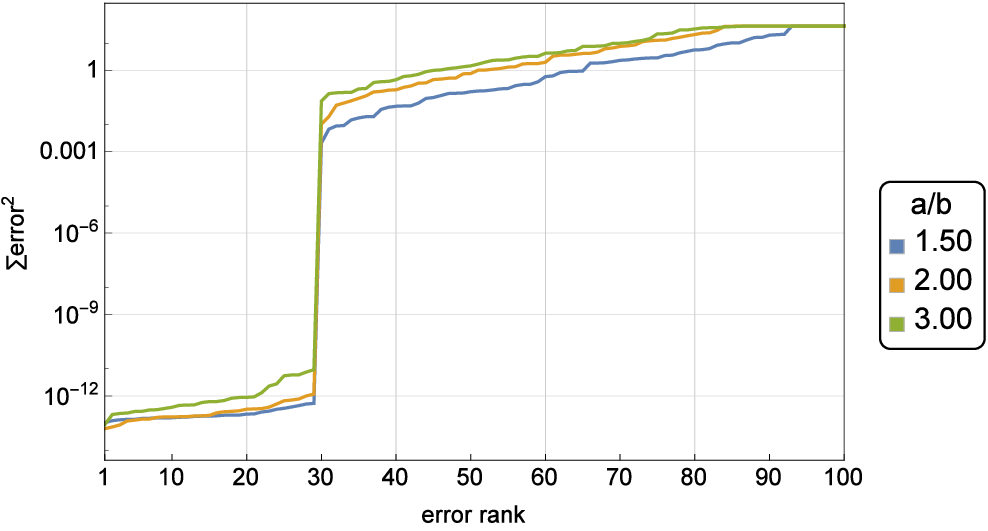}
    \caption{Log of Least-squares error for first 100 Kimberling centers in ascending order of error, for three values of $a/b$, $M=1500$. Elliptic vs. non-elliptic centers are clearly separated in two groups whose errors differ by several orders of magnitude. A table in \cite[Part II]{garcia2021-ellipses-web} shows that $X_{37}$ and $X_6$ are have ranks $30$ and $31$ respectively, i.e., they are the centers whose non-elliptic loci are closest to a perfect ellipse.}
    \label{fig:least-squares-error-graph}
\end{figure}

\begin{table}
\begin{minipage}[t]{.475\linewidth}
$$
\scriptsize
\begin{array}{|c|c|l|c|}
\hline
\text{row} & X_i & \text{definition} & \text{sim} \\
\hline
 1 & {1} & \text{Incenter} & \text{J}^t \\
 2 & {2} & \text{Barycenter} & \text{B} \\
 3 & {3} & \text{Circumcenter} & \text{C}^t  \\
 4 & {4} & \text{Orthocenter} & \text{B}^t \\
 5 & {5} & \text{9-Point Center} &  \\
 6 & {7} & \text{Gergonne Point} & \text{B}\\
 7 & {8} & \text{Nagel Point} &  \\
 8 & {10} & \text{Spieker Center} & \text{B}^t\\
 9 & {11} & \text{Feuerbach Point} & \text{C}^+ \\
 10 & {12} & \text{$\{X_{1,5}\}$-Harm.Conj. of $X_{11}$} & \\
 11 & {20} & \text{de Longchamps Point} & \\
 12 & {21} & \text{Schiffler Point} & \\
 13 & {35} & \text{$\{X_{1,3}\}$-Harm.Conj. of $X_{36}$} & \\
 14 & {36} & \text{Inverse-in-Circumc. of $X_1$} & \\
 15 & {40} & \text{Bevan Point} & \text{B}^t \\
 \hline
\end{array}
$$
\end{minipage}%
\begin{minipage}[t]{.475\linewidth}
$$
\scriptsize
\begin{array}{|c|c|l|c|}
\hline
\text{row} & X_i & \text{definition} & \text{sim} \\
\hline
16 & {46} & \text{$X_4$-Ceva Conj. of $X_1$} & \\
 17 & {55} & \text{Insimilictr(Circumc.,Incir.)} & \text{C}\\
 18 & {56} & \text{Exsimilictr(Circumc.,Incir.)} &  \\
 19 & {57} & \text{Isogonal Conj. of $X_9$} & \text{B} \\
 20 & {63} & \text{Isogonal Conj. of $X_{19}$} & \text{B} \\
 21 & {65} & \text{Intouch Triangle's $X_4$} & \\
 22 & {72} & \text{Isogonal Conj. of $X_{28}$} & \text{J}\\
 23 & {78} & \text{Isogonal Conj. of $X_{34}$} & \\
 24 & {79} & \text{Isogonal Conj. of $X_{35}$} & \\
 25 & {80} & \text{Refl. of $X_1$ about $X_{11}$} &  \text{J}^t \\
 26 & {84} & \text{Isogonal Conj. of $X_{40}$} & C^t \\
 27 & {88} & \text{Isogonal Conj. of $X_{44}$} &  \text{B}^+ \\
 28 & {90} & \text{$X_{3}$-Cross Conj. of $X_{1}$} & \\
 29 & {100} & \text{Anticomplement of $X_{11}$} & \text{B}^+\\
 \hline
 \end{array}
 $$
\end{minipage}
\caption{The 29 Kimberling centers within $X_1$ to $X_{100}$ with elliptic loci. Under column ``sim.'', letters B,C,J indicate the locus is similar to EB, caustic, or Excentral locus, respectively. An additional {+}  (resp. {t}) exponent indicates the locus is identical (resp. similar to a perpendicular copy) to the indicated ellipse. Note: the ellipticity of  $X_i$,$i=1,2,3,4$ was previously proven \cite{olga14,sergei2016proj,corentin19,garcia2019-incenter}.}
\label{tab:ell}
\end{table}

\subsection{The quartic locus of the Symmedian Point}

The construction of the Symmedian point $X_6$ of a triangle is shown in Figure~\ref{fig:x6}.

\begin{figure}
    \centering
    \includegraphics[trim=0 125 0 0,clip,width=.8\textwidth]{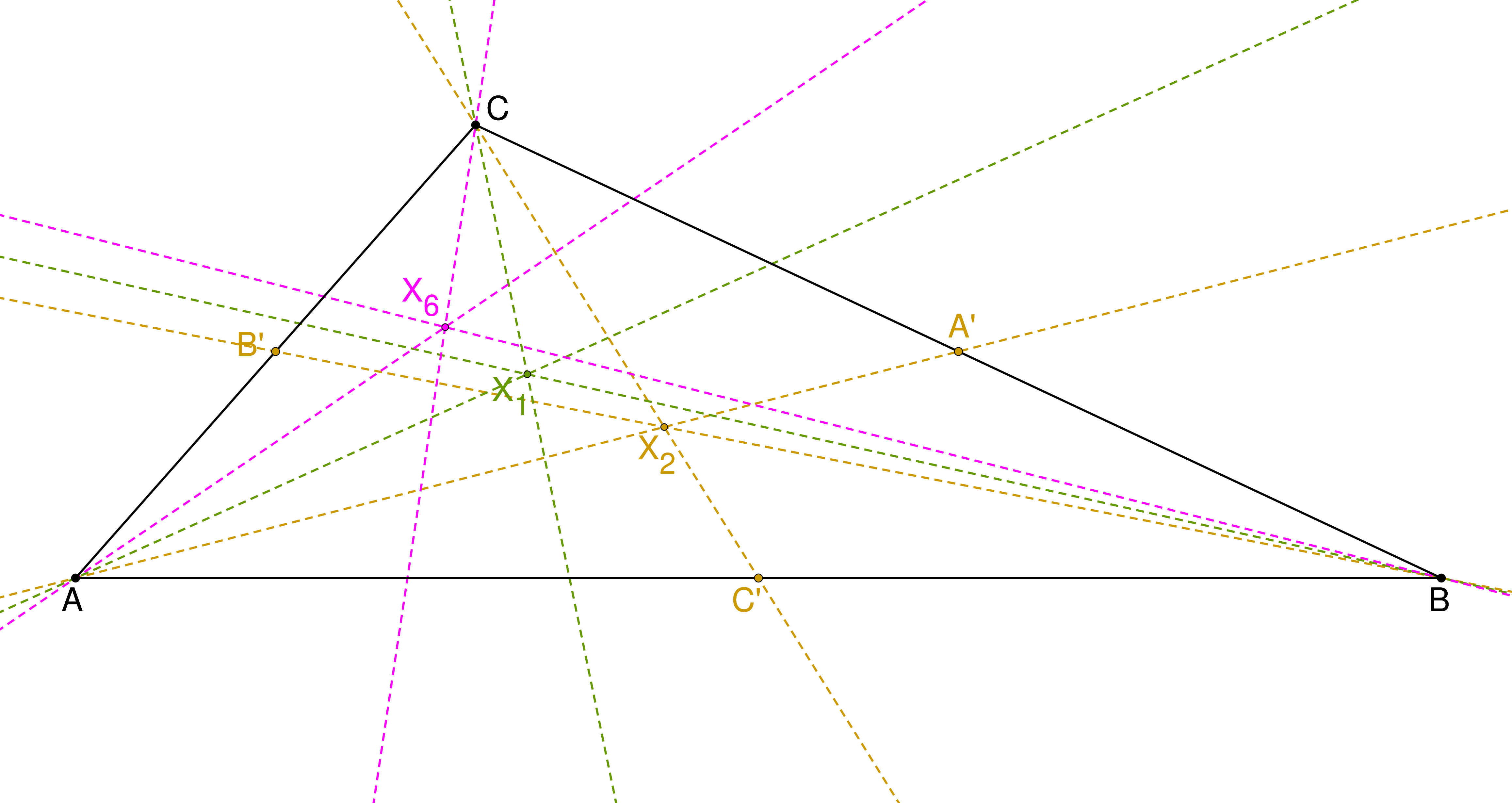}
    \caption{Given a triangle $ABC$, the Symmedian point $X_6$ is the point of concurrence of the three {\em symmedians} (pink), where the latter is a reflection of a median (brown) about the corresponding angular bisector (green). Medians (resp. angular bisectors) concur at the barycenter $X_2$ (resp.~the incenter $X_1$).}
    \label{fig:x6}
\end{figure}

Over 3-periodic orbits in an EB with $a/b=1.5$, the locus of $X_6$ is visually indistinguishable from an ellipse, Figure~\ref{fig:symmedian}. Fortunately, its fit error is 10 orders of magnitude higher than the ones produced by true elliptic loci, see \cite[Part II]{garcia2021-ellipses-web}. So it is easily rejected by the least-squares phase. Indeed, symbolic manipulation with a CAS yields:

\begin{theorem}
\label{thm:x6}
 The locus of $X_6$ is a convex quartic given by:

\begin{equation*}
  \mathcal{X}_6(x,y)=c_1 x^4+c_2 y^4+c_3 x^2 y^2+ c_4 x^2 + c_5 y^2 = 0
\end{equation*}

\noindent where:
$$
\begin{array}{rlrl}
c_1=&b^4(5\delta^2-4(a^2-b^2)\delta -a^2 b^2)&c_2=&a^4(5\delta^2+4(a^2-b^2)\delta-a^2b^2) \\
c_3=&2a^2 b^2(a^2 b^2+3\delta^2)&c_4=&a^2 b^4(3 b^4+2(2 a^2-b^2)\delta-5\delta^2)\\
c_5=&a^4 b^2(3 a^4+2(2 b^2-a^2)\delta-5\delta^2)&\delta=&\sqrt{a^4-a^2 b^2+b^4}
\end{array}
$$
\end{theorem}

\begin{proof}
Using a CAS, obtain symbolic expressions for the coefficients of a quartic symmetric about both axes (no odd-degree terms), passing through 5 known-points. Still using a CAS, verify the symbolic parametric for the locus satisfies the quartic.
\end{proof}

 \noindent Note the above is also satisfied by a degenerate level curve $(x,y)=(0,0)$, which we ignore.

\begin{remark}
The axis-aligned ellipse $\mathcal{E}_6$ with semi-axes $a_6,b_6$ is internally tangent to $\mathcal{X}_6(x,y)=0$ at the four vertices where:
{\small  
\begin{align}
a_6= \frac{\left[(3\,a^2-b^2)\delta -(a^2+b^2)b^2\right]a}{a^2b^2+3\delta^2},\;\;\;
b_6= \frac{\left[(a^2-3\,b^2)\delta + (a^2+b^2)a^2\right]b}{a^2b^2+3\delta^2}
\label{eqn:x6-ellipse}
\end{align}
}
\end{remark}

\begin{figure}
    \centering
    \includegraphics[width=.6\textwidth]{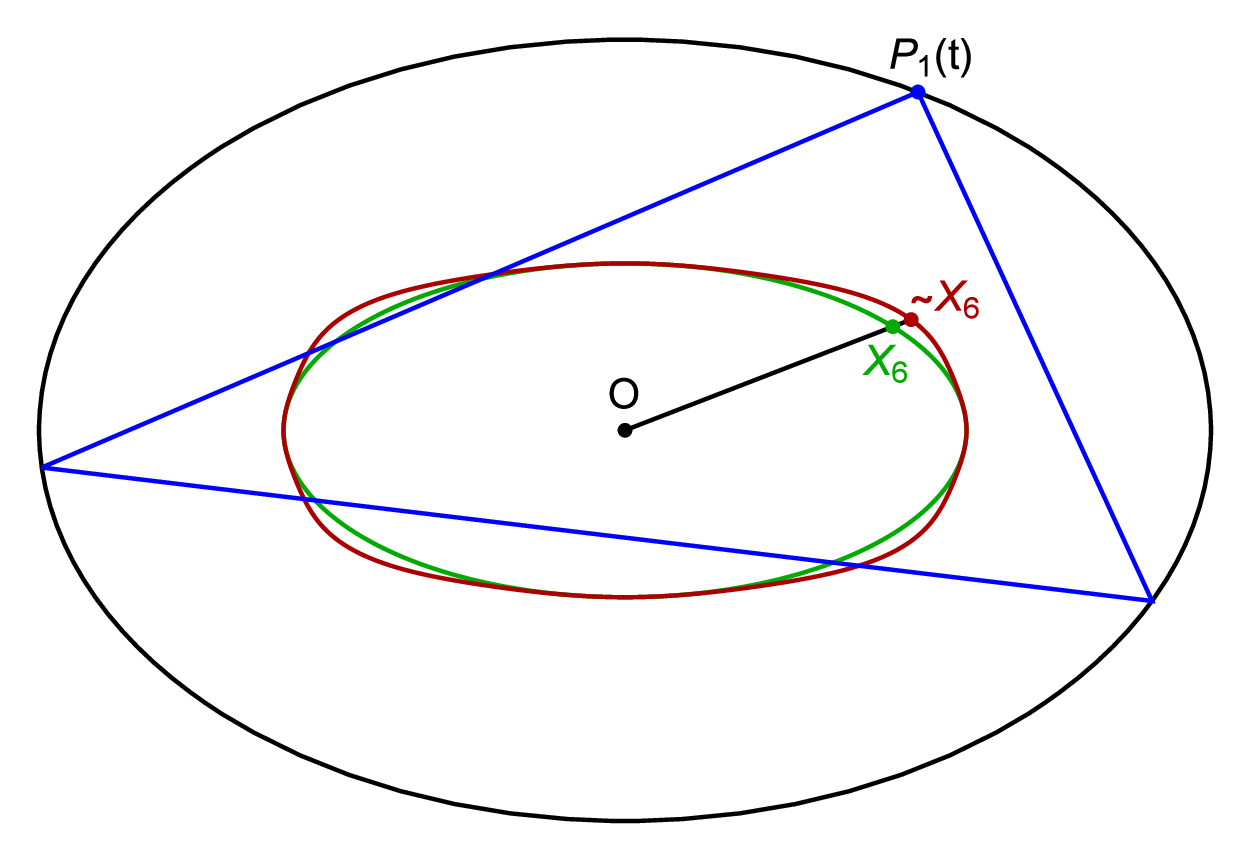}
    \caption{An $a/b=1.5$ EB is shown (black) as well as a sample 3-periodic (blue). At this aspect ratio, the locus of $X_6$ (green) is indistinguishable to the naked eye from a perfect ellipse. To see it is non-elliptic, consider the locus of a point ${\sim}X_6(t)=X_6(t)+k|X_6(t)-Y_6'(t)|$, with $k=2{\times}10^6$ and $Y_6'(t)$ the intersection of $OX_6(t)$ with a best-fit ellipse (visually indistinguishable from green), \eqref{eqn:x6-ellipse}. \href{https://bit.ly/3qc0Z0L}{app}}
    \label{fig:symmedian}
\end{figure}

Table~\ref{tab:quartic-coeffs} shows the above coefficients numerically for a few values of $a/b$.

\begin{table}[H]
    \centering
$$
\begin{array}{|c|c|c|c|c|c|c|c|}
\hline
 \text{a/b} & a_6 & b_6 & c_1/c_3 & c_2/c_3 & c_4/c_3 & c_5/c_3 & A(\mathcal{E}_6)/A(\mathcal{X}_6) \\
 \hline
  1.25 & 0.433 & 0.282 & 0.211 & 1.185 & -0.040 & -0.095 & 0.9999 \\
 1.50 & 0.874 & 0.427 & 0.114 & 2.184 & -0.087 & -0.399 & 0.9998 \\
 2.00 & 1.612 & 0.549 & 0.052 & 4.850 & -0.134 & -1.461 & 0.9983 \\
 3.00 & 2.791 & 0.620 & 0.020 & 12.423 & -0.157 & -4.769 & 0.9949 \\
 \hline
\end{array}
$$
\caption{Coefficients $c_i/c_3$, $i=1,2,4,5$ for the quartic locus of $X_6$ as well as the axes $a_6,b_6$ for the best-fit ellipse, for various values of $a/b$. The last-column reports the area ratio of the internal ellipse $\mathcal{E}_6$ (with axes $a_6,b_6$) to that of the quartic locus $\mathcal{X}_6$, showing an almost exact match.}
\label{tab:quartic-coeffs}
\end{table}

\subsection{Locus Triple Winding}
\label{sec:triple-winding}

As an illustration of Lemma~\ref{lem:center-cover}, consider the elliptic locus of $X_1$, the Incenter\footnote{The same argument is valid for the non-elliptic locus of, e.g., $X_{59}$, Figure~\ref{fig:incenter-loci} (right).}. Consider the locus of a point $Y_1$ located on between $X_1$ and an Intouchpoint $I_1$, Figure~\ref{fig:incenter-loci} (left):
 
\begin{equation*}
Y_1(t;\rho)=(1-{\rho})X_1(t)+{\rho}I_1(t),\;\;\;\rho\in[0,1]
 \end{equation*}
 
\noindent When $\rho=1$ (resp. $0$), $Y_1(t)$ is the two-lobe locus of the Intouchpoints (resp. the elliptic locus of $X_1$). With $\rho$ just above zero, $Y_1$ winds thrice around the EB center. At $\rho=0$, the two lobes and the remainder of the locus become one and the same: $Y_1$ winds thrice over the locus of the Incenter, i.e., the latter is the limit of such a convex combination.

It can be shown that at $\rho=\rho^*$, with $\rho^*=1-(b/a)^2$, the two $Y_1(t)$ lobes touch at the the EB center. When $\rho>\rho^*$ (resp. $\rho<\rho^*$), the locus of $Y_1(t)$ has winding number 1 (resp. 3) with respect to the EB center, see Figure~\ref{fig:inc-wind3}.

A similar phenomenon occurs for loci of convex combinations of the following pairs: (i) Barycenter $X_2$ and a side midpoint, (ii) Circumcenter $X_3$ and a side midpoint, (iii) Orthocenter $X_4$ and altitude foot, etc., see \cite[pl\#11,12]{dsr_playlist_2020}.

\begin{figure}
    \centering
    \includegraphics[width=\textwidth]{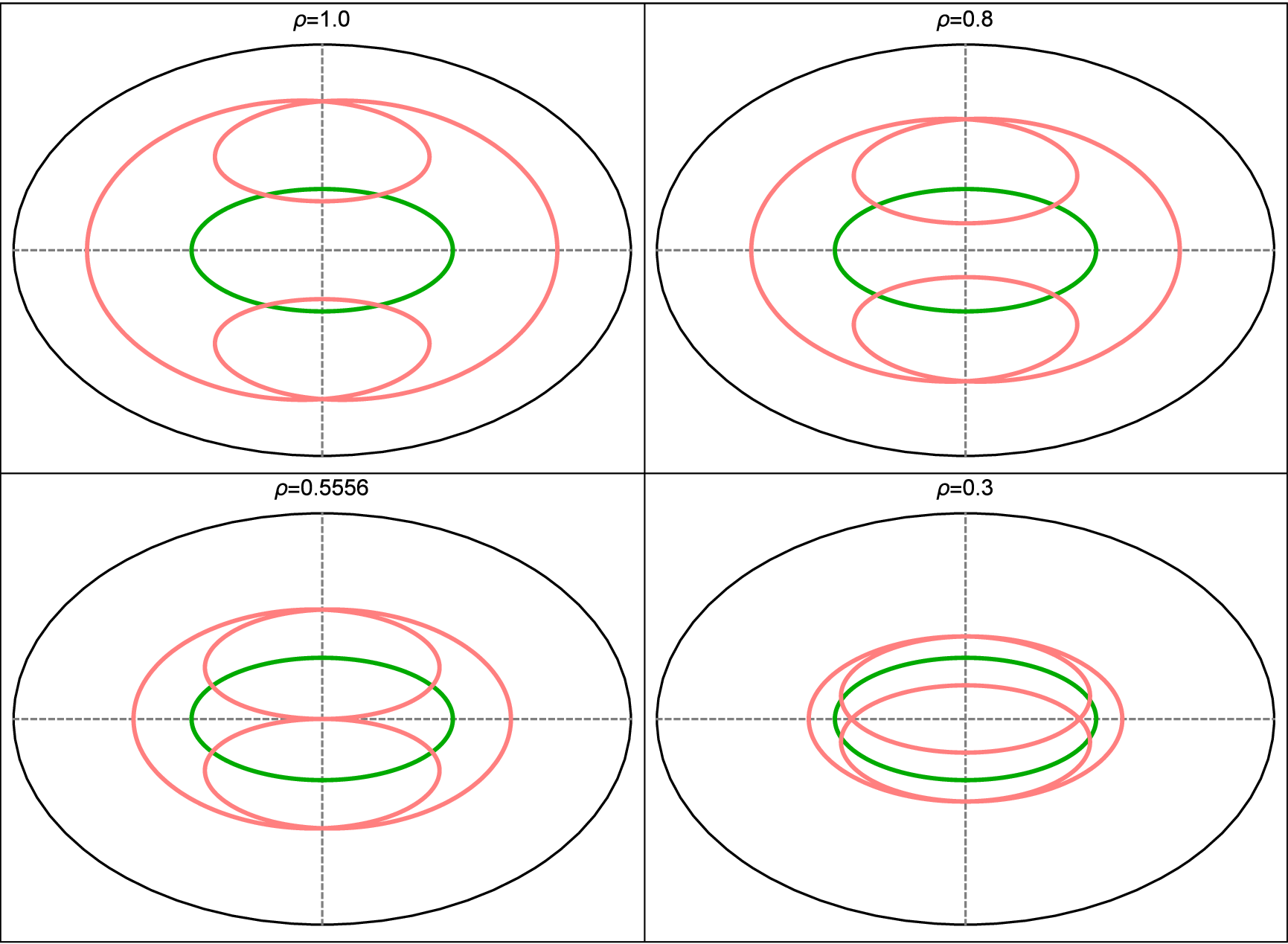}
    \caption{An $a/b=1.5$ EB is shown (black) as well as the elliptic locus of $X_1$ (green) and the locus of $Y_1(t)$ (pink), the convex combination of $X_1(t)$ and an Intouchpoint given by a parameter $\rho\in[0,1]$, see Figure~\ref{fig:incenter-loci}(left). At $\rho=1$ (top-left), $Y_1(t)$ is the two-lobe locus of the Intouchpoint. For every tour of an orbit vertex $P_1(t)$ around the EB, $Y_1(t)$ winds once over its locus. At $\rho=0.8$ (top-right) the lobes approach each other but still lie in different half planes. At $\rho=\rho^*=1-(b/a)^2$, the lobes touch at the EB center (bottom-left). If $\rho\in(0,\rho^*)$, the two lobes self-intersect twice. As $\rho{\rightarrow}0$, the two lobes become nearly coincidental (bottom-right). At $\rho=0$, the $Y_1$ locus with its two lobes all collapse to the Incenter locus ellipse (green), in such a way that for every tour of $P_1(t)$ around the EB, $X_1$ winds thrice over its locus. \href{https://youtu.be/3Gr3Nh5-jHs}{Video 1}, \href{https://youtu.be/HZFjkWD_CnE}{Video 2}}
    \label{fig:inc-wind3}
\end{figure}

\section{Toward a Typification of Loci}
\label{sec:algebraic}
Our use of a few dozen centers listed in \cite{etc} was a means to validate our approach. In general we would like to predict locus type based on any triangle function, hand-curated or not. Below we take a few steps toward building a practical Computational Algebraic Geometry context useful for practitioners.

\subsection{Trilinears: No Apparent Pattern}

When one looks at a few examples of triangle centers whose loci are elliptic vs non, one finds no apparent algebraic pattern in said trilinears, Table~\ref{tab:center-trilinears} in Appendix~\ref{app:constr}. 

A few observations include:

\begin{itemize}
\item The locus of a triangle center is symmetric about both EB axes, Lemma~\ref{lem:axisymmetric}, Section~\ref{sec:algebraic}.
\item Some trilinear centers rational on the sidelengths which produce (i) elliptic loci (e.g., $X_1,X_2$, etc.) as well as (ii) non-elliptic (e.g., $X_6$, $X_{19}$, etc.).
\item No locus has been found with more than 6 intersections with a straight line, suggesting the degree is at most 6.
\item No center has been found with non-rational trilinears whose locus is an ellipse\footnote{Not shown, but also tested were non-rational triangle centers $X_j$, $j=14,\,16,\,17,\,18,\,359,\,360,\,364,\,365,\,367$.}, suggesting that the locus of non-rational centers is always non-elliptic.
\end{itemize}

\subsection{An Algebro-Geometric Ambient}

Given EB semi-axes $a,b$, our pro\-blem can be described by the following 14 variables:

\begin{itemize}

\item  6 triangle vertex coordinates, $P_i= (x_i, y_i), \, i=1,2,3$; 

\item 3 sidelengths $s_1, s_2, s_3$;

\item 3 trilinears  $p,q,r$;
\item  2 locus coordinates $x,y$.
\end{itemize}

These are related by the system of 14 polynomial equations defined in Table~\ref{tab:zero-set}.

\begin{table}
\scriptsize
\[
\begin{array}{|c|l|l|}
\hline
\textbf{eqns.} & \textbf{description} & \textbf{zero set of} \\
\hline
3 & \text{vertices on the EB} & (x_i/a)^2 + (y_i/b)^2 - 1\;,i = 1,2,3 \\
\hline
3 & \begin{array}{l} \text{reflection law at $P_j$} \\ j,k,\ell\;\text{cyclic}, \mathcal{A}=\text{diag}(1/a^2;1/b^2)\end{array} &
\begin{array}{l}
(\mathcal{A} P_j . P_{\ell}- \mathcal{A} P_j.P_j)|P_k-P_j| \\
\;\;\;-(\mathcal{A} P_j . P_k -  \mathcal{A} P_j . P_j)|P_{\ell} - P_j|
\end{array} \\
\hline
3 & \text{sidelengths} & (x_i-x_j)^2 + (y_i-y_j)^2 -  s_k^2 \\
\hline
2 & \text{locus cartesians, \eqref{eqn:trilin-cartesian}} & (p s_1 + q s_2 + r s_3) (x, y) - p s_1 P_1 + q s_2 P_2 + r s_3 P_3 \\
\hline
3 & \text{trilinears (must rationalize)} & p - h(s_1, s_2, s_3);\; q- h(s_2, s_3, s_1);\; r- h(s_3, s_1, s_2) \\
\hline
\end{array}
\]
\caption{System of 14 equations which the locus must satisfy. The three equations in the third line are obtained from the reflection law -- angle of incidence equals the angle of reflection -- imposed on the vertices of a triangular orbit.}
\label{tab:zero-set}
\end{table}

Since the 3-periodic family of orbits is one dimensional, out of the first 9 equations, one is functionally dependent on the rest. Therefore, we have 13 independent equations in 14 variables, yielding a 1d algebraic variety, which can be complexified if desired, as in \cite{corentin19,glutsyuk2014,griffiths1978,olga14}.

Can tools from computational Algebraic Geometry \cite{Schenck2003-GA,Sturmfels97-resultants} be used to eliminate 12 variables automatically, thus obtaining a single polynomial equation $\mathcal{L}(x,y)=0$ whose Zariski closure contains the locus? Below we provide a method based on the theory of resultants \cite{lang,Sturmfels97-resultants} to compute $\mathcal{L}$ for a subset of triangle centers.

\subsection{When Trilinears are Rational}
\label{sec:rational-trilinears}

Consider a triangle center $X$ whose trilinears $p:q:r$ are rational on the sidelengths $s_1,s_2,s_3$, i.e., the triangle center function $h$ is rational, equation \eqref{eqn:ftrilins}.

\begin{theorem}
The locus of a rational triangle center is an algebraic curve.
\label{thm:rational-center}
\end{theorem}

We thank one of the referees for kindly contributing the outline for this proof.
 
\begin{proof}
Consider the complexified variables of the first 4 rows of the table \ref{tab:zero-set}: the $x_i, y_i$ and $s_i$. We have a 9 dimensional complex space. Consider the Zariski closure $V$ in $\cp 9$ of the complex algebraic subset defined by the 9 first equations given in the Table \ref{tab:zero-set}.
Explicitly we have the equations:
{\small  
\begin{align*}
    f_i&=\frac{x_i^2}{a^2}+\frac{y_i^2}{b^2}-1=0, \; (i=1,2,3)\\
    f_4&= (\frac{x_1 x_2 }{ a^2} + \frac{y_1 y_2 }{b^2}  - 1)s_2 - (\frac{x_1 x_3}{ a^2} +\frac{ y_1 y_3}{b^2 }- 1) s_3=0\\
    f_5&=(\frac{x_1 x_2 }{ a^2} + \frac{y_1 y_2 }{b^2}  - 1)s_1 - (\frac{x_2 x_3}{ a^2} +\frac{ y_2 y_3}{b^2 }- 1) s_3=0\\
    f_6&=(\frac{x_1 x_3 }{ a^2} + \frac{y_1 y_3 }{b^2}  - 1)s_1 - (\frac{x_2 x_3}{ a^2} +\frac{ y_2 y_3}{b^2 }- 1) s_2=0\\
    f_7&= (x_2 - x_1)^2 + (y_2 - y_1)^2 - s_3^2=0 \\
    f_8&=  (x_2 - x_3)^2 + (y_2 - y_3)^2 - s_1^2=0 \\
    f_9&=(x_3 - x_1)^2 + (y_3 - y_1)^2 - s_2^2=0\\
\end{align*}
}

By construction, it is a complex projective algebraic subset of $\cp 9$. We have that $ s_2f_5-s_3f_6-s_1f_4=0$. Computing its dimension it follows that:
 
\begin{enumerate}[(i)]

\item the rank of the Jacobian matrix given by the 5 equations $f_1=f_2=f_3=f_4=f_5=0$ in relation to the variables $(x_1,y_1,x_2,y_2,x_3) $ or $(x_1,y_1,x_2,y_2,y_3) $ is generically 5. 

\item the rank of the Jacobian matrix given by the 9 equations is generically 8. 
\item the $s_k$ are  determined by the $x_i, y_i$ (last 3 equations of Table~\ref{tab:zero-set}) and that the set of complex triangular orbits is of dimension 1, see \cite{corentin19,glutsyuk2014,griffiths1978,olga14}, and \cite{garcia2019-incenter} for an explicit parametrization of $V$. Longer (explicit) expressions  appear in Appendix~\ref{app:rational-support}.

\end{enumerate}

This implies that $V$ is a projective algebraic curve of $\cp 9$, the complex projective space of dimension 9. The map which associates any point of $V$ with $X=(x,y)\in\cp 2$ as in Equation~\ref{eqn:trilin-cartesian} (where $p, q, r$ are given rational maps of the $s_i$) is well-defined and holomorphic on an algebraic open subset of $V$, and hence, on the complement of a discrete subset of $V$. Therefore, it is defined and holomorphic everywhere on $V$. Now its image is an analytic subset of $\cp 2$ (Remmert proper mapping theorem) hence an algebraic subset (Chow theorem), see \cite[Ch.V]{gunning1965}. Its dimension is less or equal than 1, otherwise it would be the whole $\cp 2$ (an impossibility, since it is bounded for real points). 
When the dimension of the image is zero the locus degenerates to a point. \end{proof}

A typical example of a point-locus is that of the Mittenpunkt $X_9$, which remains stationary at the center of the EB. In this case $p=h(s_1,s_2,s_3)=s_2+s_3-s_1$, $q=h(s_2,s_3,s_1)=s_1+s_3-s_2$ and $r=h(s_3,s_1,s_2)=s_1+s_2-s_3$. In the EB we have that $X_9=(0,0)$. 

\begin{proposition}
Over 3-periodics in the EB, the Mittenpunkt $X_9$ is the only triangle center whose locus is a point.
\end{proposition}

\begin{proof}
Given the four-fold symmetry in the one-dimensional family of 3-periodics in the EB, the locus of any triangle center must be symmetric about both the vertical and horizontal semi-axes of the EB. So if some center's locus is a point, said point must be at the center of the EB. Over a continuous set of generic triangles (such as those in the 3-periodic family), two distinct triangle centers cannot always coincide (other than in a discrete set of configurations), so the claim follows.
\end{proof}

\subsection{Algorithm to compute the locus polynomial}
Our algorithm is based on the following 3-steps which yield an algebraic curve $\mathcal{L}(x,y)=0$ which contains the locus. We refer to Lemmas \ref{lem:1coord} and \ref{lem:2sides} appearing below. Appendix~\ref{app:rational-support} contains  supporting expressions.

\begin{step}
Introduce the symbolic variables $u, u_1, u_2$:
\begin{equation*}
    u^2 + u_1^2 = 1,\;\;\;\rho_1\, u^2 + u_2^2 = 1.
\end{equation*} 
\end{step}
 \noindent The vertices will be given by rational functions of   $u, u_1, u_2$ 
\begin{equation*} P_1 = (a\,u, b\,u_1),\;\;P_2 = (p_{2x}, p_{2y})/q_2,\;\;\;P_3 = (p_{3x}, p_{3y})/q_3 
\end{equation*}
 
\noindent Expressions for $P_1,P_2,P_3$ appear in Appendix \ref{app:rational-support} as do equations $g_i=0$, $i=1,2,3$, polynomial in $ s_i,u,u_1,u_2$.
 
\begin{step}Express the locus  $X$ as a  rational function on  $u,u_1, u_2, s_1, s_2, s_3$.
\end{step}

Convert $p:q:r$ to cartesian coordinates $X = (x,y)$ via Equation~\eqref{eqn:trilin-cartesian}. From Lemma~\ref{lem:1coord}, it follows that
$\left(x,y\right)$ is rational on $u,u_1,u_2,s_1,s_2,s_3$.
\begin{equation*} x=\mathcal{Q}/\mathcal{R},\;\;\;y=\mathcal{S}/\mathcal{T}
\end{equation*}

\noindent To obtain the polynomials    $\mathcal{Q,R,S,T}$  on said variables $u,u_1,u_2,s_1,s_2,s_3$,
 one substitutes the 
$p,q,r$ by the corresponding rational functions of  $s_1, s_2, s_3$ that define a specific triangle center $X$. Other than that, the method proceeds identically.

\begin{step}
Computing resultants.
Our problem is now cast in terms of the polynomial equations:
\begin{equation*}
E_0= \mathcal{Q}-x\,\mathcal{R}=0,\;\;\; F_0= \mathcal{S}-y\,\mathcal{T}=0
\end{equation*}

\end{step}

Firstly, compute the resultants, in chain fashion:  
\begin{align*}
    E_1=&\textrm{Res}(g_1,E_0,s_1)=0,\;\;\;F_1=\textrm{Res}(g_1,F_0,s_1)=0\\
	E_2=&\textrm{Res}(g_2,E_1,s_2)=0,\;\;\;F_2=\textrm{Res}(g_2,F_1,s_2)=0\\
	E_3=&\textrm{Res}(g_3,E_2,s_3)=0,\;\;\;F_3=\;\textrm{Res}(g_3,F_2,s_3)=0
\end{align*}
		 
It follows that  $E_3(x,u,u_1,u_2)=0$ and $F_3(y,u,u_1,u_2)=0$ are polynomial
equations. In other words, $s_1, s_2, s_3$ have been eliminated. 

Now  eliminate the variables $u_1$ and $u_2$ by taking the following resultants:
\begin{align*}
	E_4(x,u,u_2)=&\textrm{Res}(E_3,u_1^2+u^2-1,u_1)=0\\ 	F_4(y,u,u_2)=&\textrm{Res}(F_3,u_1^2+u^2-1,u_1)=0\\
	E_5(x,u)=&\textrm{Res}(E_4,u_2^2+\rho_1 u^2-1,u_2)=0\\
	F_5(y,u)=&\textrm{Res}(F_4,u_2^2+\rho_1 u^2-1,u_2)=0
\end{align*}

This yields two polynomial equations $E_5(x,u)=0$ and $F_5(y,u)=0$. 

Finally compute the resultant
\[{\mathcal L} = \textrm{Res}(E_5,F_5,u)=0\]
that eliminates $u$ and gives  the implicit algebraic equation for the locus $X$. 

\begin{remark}
In practice,  after  obtaining  a resultant, a human assists the CAS by factoring out spurious branches
(when recognized), in order to get the final answer in more reduced form.   
\end{remark}

When non-rational in the sidelengths, except a few cases (e.g., $X_{359}, X_{360}$ which are transcendental), triangle centers
in Kimberling's list have explicit trilinears involving fractional powers and/or terms containing the triangle area. Those can be made implicit, i.e,
given by zero sets of polynomials involving $p,q,r, s_1, s_2, s_3$.  The chain of resultants to be computed will be increased by three, in order to eliminate the variables $p,q, r$ before (or after) $s_1, s_2, s_3$.


\begin{lemma}
\label{lem:1coord}
Let $P_1=({a}{u},b\sqrt{1-u^2}).$
	The coordinates of $P_2$ and $P_3$ of the 3-periodic billiard orbit are rational functions in the variables $u, u_1, u_2$, where
	$u_1=\sqrt{1-u^2}$, $u_2=\sqrt{1-\rho_1 u^2}$ 
and
	$\rho_1=c^4(b^2+\delta)^2/a^6$. 
		
	\end{lemma}
	
	\begin{proof}
	Follows directly from the parametrization of the billiard orbit, Appendix~\ref{app:p1p2p3}. In fact,  $P_2=(x_2(u),y_2(u)) =( p_{2x}/q_2, p_{2y}/q_2)$ and $P_3=(x_3(u),y_3(u))$ $=( p_{3x}/q_3, p_{3y}/q_3)$, where $p_{2x}$, $p_{2y}$, $p_{3x}$ and $p_{3y}$ have degree $4$ in $(u,u_1,u_2)$  and $q_2$, $q_3$ are algebraic of degree $4$ in $u$. Expressions for $u_1,u_2$ appear in Appendix~\ref{app:exit-angle}.
\end{proof}
	
\begin{lemma}
\label{lem:2sides} Let $P_1=(a u,b\sqrt{1-u^2}).$ Let $s_1$, $s_2$ and $s_3$ the sides of the triangular orbit ${P_1}{P_2}{P_3}$. Then $g_1(u,s_1)=0$, $g_2(s_2,u_2,u)=0$ and $g_3(s_3,u_2,u)=0$ for polynomial functions $g_i$, defined in  Appendix~\ref{app:alg_locus}.  
\end{lemma}
	
\begin{proof}
Using the parametrization of the 3-periodic billiard orbit it follows that $s_1^2-|P_2-P_3|^2=0$ is a rational equation in the variables $u,s_1$. Simplifying, leads to $g_1(s_1,u)=0.$

Analogously for $s_2$ and $s_3$. In this case, the equations $s_2^2-|P_1-P_3|^2=0$ and  $s_3^2-|P_1-P_2|^2=0$   have   square roots $u_2=\sqrt{1-\rho_1 u^2}$ and $u_1=\sqrt{1-u^2}$ and  are rational in the variables $s_2,u_2,u_1,u$ and $s_3,u_2,u_1,u$ respectively. It follows that the degrees of $g_1$, $g_2$, and $g_3$ are $10$. Simplifying, leads to $g_2(s_2,u_2,u_1,u)=0 $ and $g_3(s_3,u_2,u_1,u)=0$. 
\end{proof}

\subsection{Examples}

Table~\ref{tab:zariski} shows the Zariski closure obtained contained the elliptic locus of a few triangle centers. Notice one factor is always of the form $[(x/a_i)^2+(y/b_i)^2)-1]^3$, related to the triple cover described in Section~\ref{sec:triple-winding}. The expressions shown required some manual simplification during the symbolic calculations. 

\begin{table}
$$
\scriptsize
\begin{array}{|c|l|l|l|l|l|}
\hline
X_i & \text{Name} & \text{Spurious Factors} & \text{Elliptic Factor} & a_i & b_i \\
\hline
1 & \text{Incenter} & \text{very long expression} & \begin{array}{l}
[36 (91 + 61 \sqrt{61}) x^2 \\
+ 324 (139 + 19 \sqrt{61}) y^2 \\
+ 22761 - 3969 \sqrt{61}]^6
\end{array} & 0.63504 & 0.29744 \\
\hline
2 & \text{Barycenter} & \begin{array}{l} (91500x^2\\
+49922\sqrt{61}-370993)^2 \end{array} & \begin{array}{l} (100x^2\\+225y^2\\
+52\sqrt{61}-413)^3 \end{array} & 0.26205 & 0.1747 \\
\hline
3 & \text{Circumcenter} & \begin{array}{l}
x\\(3600 x^4-6380 x^2-1539)\\
(40 x^2+5\sqrt{61}-43)^2\\
\begin{array}{l}
(-1104500 y^2\\+591136\sqrt{61}-4633685)^2\end{array}\\
(65880x^2+8527\sqrt{61}-64649)^6 \end{array} & \begin{array}{l}
(5832 x^2\\
+(2752-320\sqrt{61}) y^2\\5751-729\sqrt{61})^3 \end{array} & 0.099146 & 0.4763\\
\hline
\end{array}
$$
\caption{Method of Resultants applied to obtain the zero set for a few sample triangle centers with elliptic loci, for the specific case of $a/b=1.5$. Both spurious and an elliptic factor are present. The latter are raised to powers multiple of three suggesting a phenomenon related to the triple cover, Section~\ref{sec:triple-winding}. Also shown are semi-axes $a_i,b_i$ implied by the elliptic factor. These have been checked to be in perfect agreement with the values predicted for those semi-axes in \cite{garcia2019-incenter}.}
\label{tab:zariski}
\end{table}

\section{Conclusion}
\label{sec:conclusion}
A few interesting questions are posed to the reader.

\begin{itemize}
    \item Can the degree of the locus of a triangle center or derived triangle vertex be predicted based on its trilinears?
    \item Is there a triangle center such that its locus intersects a straight line more than 6 times?
    \item Certain triangle centers have non-convex loci (e.g., $X_{67}$ at $a/b=1.5$ \cite{dsr_locus_gallery_2019}). What determines non-convexity?
    \item What determines the number of self-intersections of a given locus?
    \item In the spirit of \cite{corentin19,olga14}, how would one determine via complex analytic geometry, that $X_6$ is a quartic?
    \item What is the non-elliptic locus described by the summits of equilaterals erected over each orbit side (used in the construction of the Outer Napoleon Triangle \cite{mw}).  \cite[pl\#13]{dsr_playlist_2020}. What kind of curve is it?
    \item Within $X_1$ and $X_{100}$ only $X_i$, $i=13{\ldots}18$ have irrational trilinears. $X_j$, $j=359,\,360$ are transcendental and $X_k$, $k=364,\,365,\,367$ are irrational. In the latter two cases, the loci are numerically non-elliptic. Can any non-rational triangle center produce an ellipse?
    \item $X_6$ is the isogonal conjugate\footnote{Given a point $P$, reflect the three $P$-cevians (lines through vertices and point $P$) about the angular bisectors. These meet at the {\em isogonal conjugate} of $P$.} of $X_2$. Though the latter's locus is an ellipse, the former's is a quartic. In the case of the isogonal pair $X_3$ and $X_4$ both are ellipses. What is the connection with isogonal (and/or isotomic\footnote{Given vertices $P_i$, $i=1,2,3$, of a triangle and a point $X$, let $Q_i$ denote the intersections of $X$-cevians with the opposite side. Let $Q_i'$ be the reflection of $Q_i$ about the midpoint of the corresponding side. Lines $P_i Q_i'$ meet at $X'$, the {\em isotomic conjugate} of $X$.} transformations) and ellipticity?
\end{itemize}

\subsection{Videos and Media}

The reader is encouraged to browse our companion paper \cite{reznik2020-ballet} where intriguing locus phenomena are investigated. Additionally, loci can be explored interactively with our browser-based  \href{https://dan-reznik.github.io/ellipse-mounted-loci-p5js/}{app} \cite{darlan2020-ellipse-mounted}.

Videos mentioned herein are on a \href{https://bit.ly/2REOigc}{playlist} \cite{dsr_playlist_2020}, with links provided on Table~\ref{tab:playlist}.

\begin{table}
\small
\begin{tabular}{|c|l|l|l|}
\hline
{id} & Title & Section & \texttt{youtu.be/<.>}\\
\hline
{01} &
{Locus of $X_1$ is an Ellipse} & \ref{sec:intro} & \href{https://youtu.be/BBsyM7RnswA}{\texttt{BsyM7RnswA}} \\
{02} &
{Locus of Intouchpoints is non-elliptic} & \ref{sec:intro}, \ref{app:early} & \href{https://youtu.be/9xU6T7hQMzs}{\texttt{9xU6T7hQMzs}}\\
{03} &
{$X_9$ stationary at EB center} & \ref{sec:intro} & \href{https://youtu.be/tMrBqfRBYik}{\texttt{tMrBqfRBYik}} \\
{04} &
{Stationary Excentral Cosine Circle}  &
\ref{sec:intro} & \href{https://youtu.be/ACinCf-D_Ok}{\texttt{ACinCf-D\_Ok}} \\
{05} &
{Loci for $X_1\ldots{X_5}$ are ellipses} &
\ref{app:early} & \href{https://youtu.be/sMcNzcYaqtg}{\texttt{sMcNzcYaqtg}} \\
{06} &
{Elliptic locus of Excenters similar to rotated $X_1$} &
\ref{app:early} & \href{https://youtu.be/Xxr1DUo19_w}{\texttt{Xxr1DUo19\_w}} \\
{07} &
{Loci of $X_{11}$, $X_{100}$ and Extouchpoints are the EB} &\ref{app:early} & \href{https://youtu.be/TXdg7tUl8lc}{\texttt{TXdg7tUl8lc}} \\
{08} &
{Family of Derived Triangles} &
\ref{sec:loci_geom} & \href{https://youtu.be/xyroRTEVNDc}{\texttt{xyroRTEVNDc}} \\
{09} &
{Loci of Vertices of Derived Triangles} &
 \ref{app:early},\ref{sec:loci_geom}& \href{https://youtu.be/OGvCQbYqJyI}{\texttt{OGvCQbYqJyI}} \\
{10} &
{Peter Moses' 29 Billiard Points} &
\ref{sec:loci_geom} & \href{https://youtu.be/JdcJt5PExsw}{\texttt{JdcJt5PExsw}}\\
{11} &
{Convex Comb.: $X_1$-Intouch and $X_2$-Midpoint} &
\ref{sec:triple-winding} & \href{https://youtu.be/3Gr3Nh5-jHs}{\texttt{3Gr3Nh5-jHs}}\\
{12} &
{Convex Comb.: $X_3$-Midpoint and $X_4$-Altfoot} &
\ref{sec:triple-winding} & \href{https://youtu.be/HZFjkWD_CnE}{\texttt{HZFjkWD\_CnE}}\\
{13} &
{Oval Locus of the Outer Napoleon Summits} &
\ref{sec:conclusion} & \href{https://youtu.be/70-E-NZrNCQ}{\texttt{70-E-NZrNCQ}} \\
\hline
\end{tabular}
\caption{Videos mentioned in the paper available in a Youtube playlist \cite{dsr_playlist_2020}. The last column contains a clickable YouTube code.}
\label{tab:playlist}
\end{table}

\section*{Acknowledgments}
We warmly thank Clark Kimberling, Peter Moses, Sergei Tabachnikov, Richard Schwartz, Arseniy Akopyan, Olga Romaskevich, Ethan Cotterill, for their invaluable input during this research. We would like to thank the referees for all very useful suggestions, including an alternative proof for Theorem 3.

The first author is fellow of CNPq and coordinator of Project PRONEX/ CNPq/ FAPEG 2017 10 26 7000 508.

\appendix
\section{Triangle Centers}
\label{app:triangle-centers}
\subsection{Triangle Centers}
\label{app:constr}

Constructions for a few basic triangle centers are shown in Figure~\ref{fig:constructions}, using Kimberling's $X_i$ notation \cite{etc}.

\begin{figure}
    \centering
    \includegraphics[width=\textwidth]{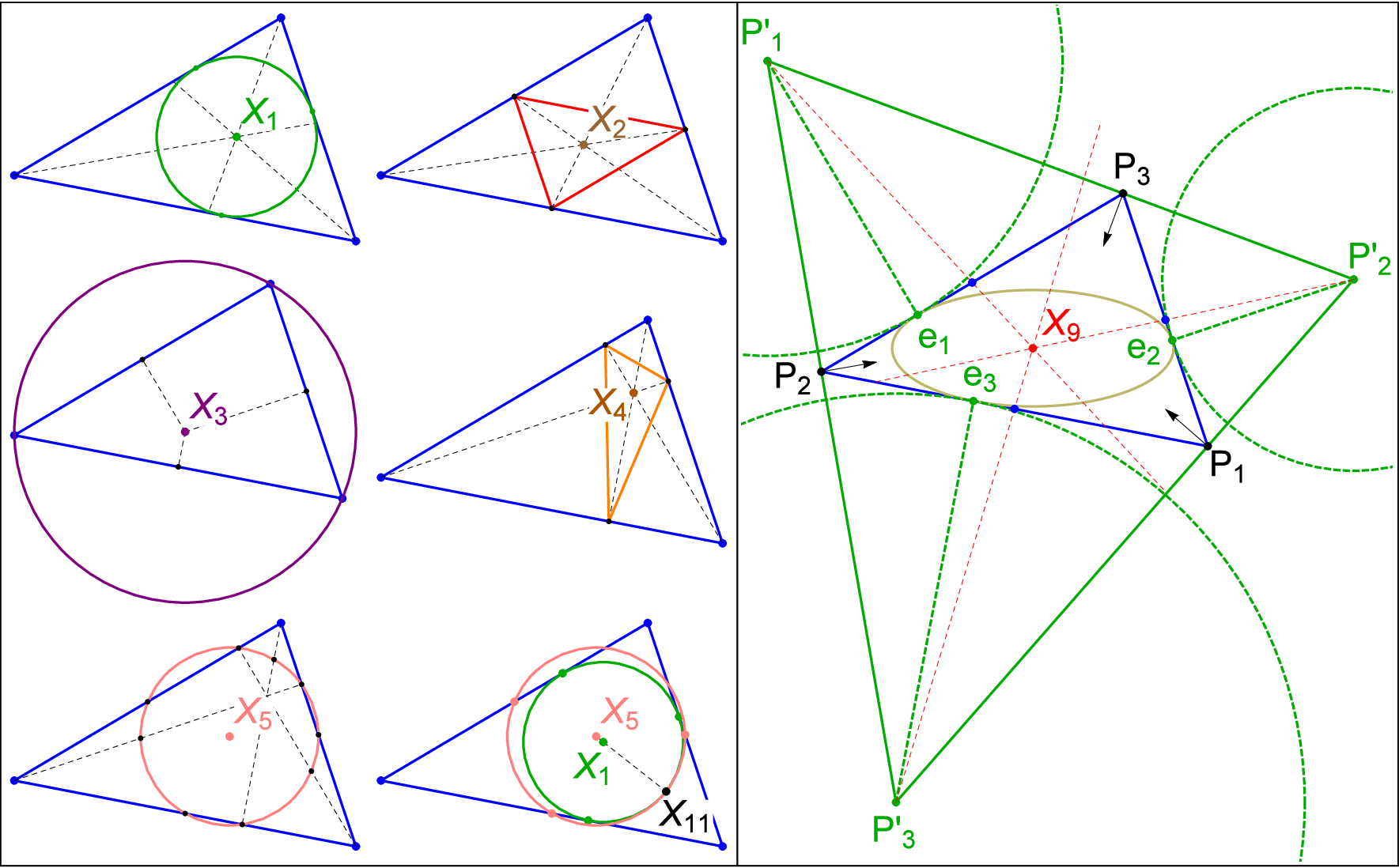}
    \caption{Constructions for triangle centers $X_i$, $i=1,2,3,4,5,9,11$, taken from \cite{reznik2020-intelligencer}.}
    \label{fig:constructions}
\end{figure}

\begin{itemize}
    \item The Incenter $X_1$ is the intersection of angular bisectors, and center of the Incircle (green), a circle tangent to the sides at three {\em Intouchpoints} (green dots), its radius is the {\em Inradius} $r$.
    \item The Barycenter $X_2$ is where lines drawn from the vertices to opposite sides' midpoints meet. Side midpoints define the {\em Medial Triangle} (red).
    \item The Circumcenter $X_3$ is the intersection of perpendicular bisectors, the center of the {\em Circumcircle} (purple) whose radius is the {\em Circumradius} $R$.
    \item The Orthocenter $X_4$ is where altitudes concur. Their feet define the {\em Orthic Triangle} (orange).
    \item $X_5$ is the center of the 9-Point (or Euler) Circle (pink): it passes through each side's midpoint, altitude feet, and Euler points \cite{mw}.
    \item The Feuerbach Point $X_{11}$ is the single point of contact between the Incircle and the 9-Point Circle.
    \item Given a reference triangle $P_1P_2P_3$ (blue), the {\em Excenters} $P_1'P_2'P_3'$ are pairwise intersections of lines through the $P_i$ and perpendicular to the bisectors. This triad defines the {\em Excentral Triangle} (green).\
    \item The {\em Excircles} (dashed green) are centered on the Excenters and are touch each side at an {\em Extouch Point} $e_i,i=1,2,3$.
    \item Lines drawn from each Excenter through sides' midpoints (dashed red) concur at the {\em Mittenpunkt} $X_9$.
    \item Also shown (brown) is the triangle's {\em Mandart Inellipse}, internally tangent to each side at the $e_i$, and centered on $X_9$. This is identical to the $N=3$ caustic.
\end{itemize}

\subsection{Trilinear Coordinates}
\label{app:trilin}

Trilinear coordinates are signed distances to the sides of a triangle $T=P_1P_2P_3$, i.e. they are invariant with respect to similarity/reflection transformations, see Figure~\ref{fig:trilins}.

\begin{figure}
    \centering
    \includegraphics[width=\textwidth]{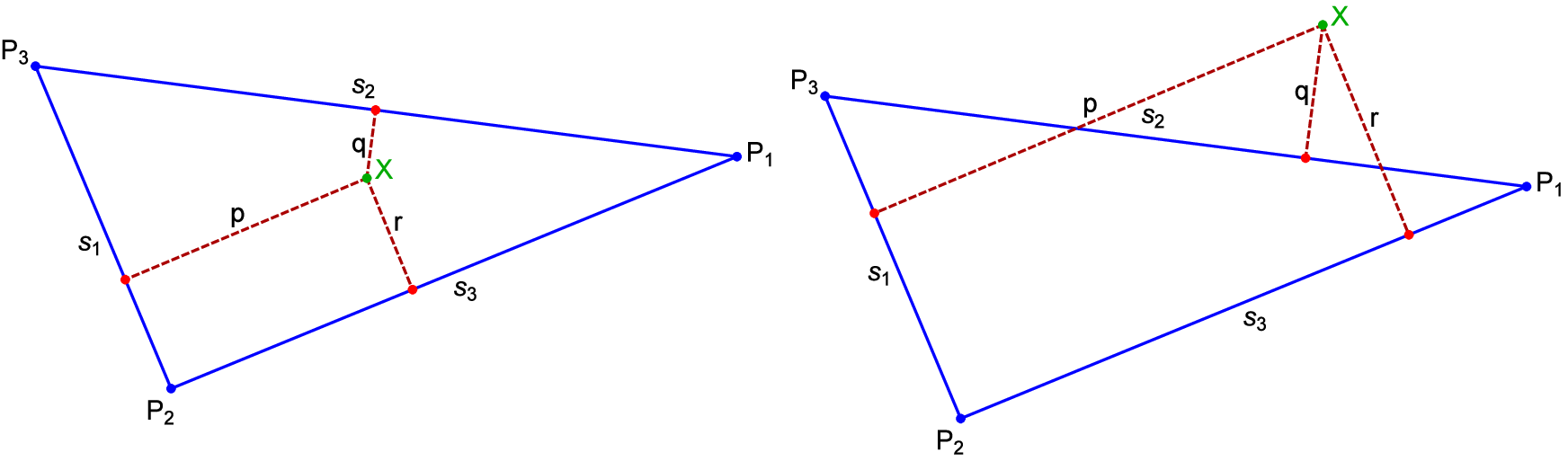}
    \caption{Trilinear coordinates $p:q:r$ for a point $X$ on the plane of a generic triangle $T=P_1P_2P_3$ are homogeneous signed distances to each side whose lengths are $s_1,s_2,s_3$. The red dots are known as the {\em pedal points} of $X$ \cite{mw}. The left (resp. right) figure shows an interior (resp. exterior) point. In both cases $p,r$ are positive however $q$ is positive (resp. negative) in the former (resp. latter) case.}
    \label{fig:trilins}
\end{figure}

Trilinears can be easily converted to cartesian coordinates using \cite{mw}:
\begin{equation}
\label{eqn:trilin-cartesian}
X =\frac{p s_1 P_1 + q s_2 P_2 + r s_3 P_3}{p{s_1}+q{s_2}+r{s_3}}
\end{equation}

\noindent where $s_1=|P_3-P_2|$, $s_2=|P_1-P_3|$ and $s_3=|P_2-P_1|$ are the sidelengths, Figure~\ref{fig:trilins}.
For instance, the trilinears for the barycenter are $p = s_1^{-1},\,q = s_2^{-1},\,r = s_3^{-1}$, yielding the familiar $X_2 = (P_1+P_2+P_3)/3$.


A {\em triangle center} is such a triple obtained by applying a {\em triangle center function} $h$ thrice to the sidelengths $s_1,s_2,s_3$ cyclically \cite{kimberling1993_rocky}:
 \begin{equation}
\label{eqn:ftrilins}
    p:q:r {\iff} h(s_1,s_2,s_3):h(s_2,s_3,s_1):h(s_3,s_1,s_2)
\end{equation}

\noindent $h$ must (i) be {\em bi-symmetric}, i.e., $h(s_1,s_2,s_3)=h(s_1,s_3,s_2)$, and (ii) homogeneous, $h(t s_1, t s_2, t s_3)=t^n h(s_1,s_2,s_3)$ for some $n$ \cite{kimberling1993_rocky}.

See Table~\ref{tab:center-trilinears} for triangle center functions for selected centers. Trilinears can be converted to cartesian coordinates using \eqref{eqn:trilin-cartesian}.

Trilinear coordinates for selected triangle centers appear in Table~\ref{tab:center-trilinears}.

\begin{table}
\scriptsize
\begin{tabular}{|c|l|l|}
\hline
center & name & $h(s_1,s_2,s_3)$ \\
\hline
$X_{1}$ & Incenter & $1$ \\
$X_{2}$ & Barycenter & $1/s_1$  \\
$X_{3}$ & Circumcenter & $s_1(s_2^2+s_3^2-s_1^2)$  \\
$X_{4}$ & Orthocenter & $1/[s_1(s_2^2+s_3^2-s_1^2)]$ \\
$X_{5}$ & 9-Point Center & ${s_2}{s_3}[s_1^2(s_2^2+s_3^2)-(s_2^2-s_3^2)^2]$ \\
$X_{11}$ & Feuerbach Point &  ${s_2}{s_3}(s_2+s_3-s_1)(s_2-s_3)^2$ \\
$X_{88}$ & Isog. Conjug. of $X_{44}$ & $1/(s_2+s_3-2{s_1})$ \\
$X_{100}$ & Anticomplement of $X_{11}$ & $1/(s_2-s_3)$  \\ 
\hline
$\mathbf{X_{6}}$ & \textbf{Symmedian Point} & $\mathbf{s_1}$  \\
$\mathbf{X_{13}^*}$ & \textbf{Fermat Point} & $\mathbf{s_1^4 - 2(s_2^2 - s_3^2)^2 + s_1^2(s_2^2 + s_3^2 + 4\sqrt{3}A)}$ \\
$\mathbf{X_{15}^*}$ & \textbf{2nd Isodynamic Point} & $\mathbf{s_1[\sqrt{3}(s_1^2 - s_2^2 - s_3^2) - 4 A]}$ \\
$\mathbf{X_{19}}$ & \textbf{Clawson Point} & $\mathbf{1/(s_2^2 + s_3^2 - s_1^2)}$ \\ 
$\mathbf{X_{37}}$ & \textbf{Crosspoint of $\mathbf{X_{1},X_{2}}$} & $\mathbf{s_2+s_3}$ \\
$\mathbf{X_{59}}$ & \textbf{Isog. Conj. of $\mathbf{X_{11}}$} & $\mathbf{1/[{s_2}{s_3}(s_2+s_3-s_1)(s_2-s_3)^2]}$ \\
\hline
$X_{9}$ & \text{Mittenpunkt} & $s_2+s_3-s_1$ \\
\hline
\end{tabular}
\caption{Triangle center function $h$ for a few selected $X_i$'s, taken from \cite{etc}. The first 8 centers produce elliptic loci, whereas the remainder (\textbf{boldfaced}) do not. $X_{13}$ and $X_{15}$ are {\em starred} to indicate their trilinears are irrational: these contain $A$, the area the triangle, known (e.g., from Heron's formula) to be irrational on the sidelengths. We haven't yet detected an algebraic pattern which differentiates both groups, nor have we detected an irrational center whose locus is elliptic. Regarding the last row, the Mittenpunkt, we don't consider its locus to be elliptic since it degenerates to a point at the EB center.}
\label{tab:center-trilinears}
\end{table}

\subsection{Derived Triangles}
\label{app:derived-tris}

A {\em derived triangle} $T'$ is constructed from vertices of a reference triangle $T$. These can be represented as a 3x3 matrix, where each row, taken as trilinears, is a vertex of $T'$. For example, the Excentral, Medial, and Intouch Triangles $T'_{exc}$, $T'_{med}$, and $T'_{int}$  are given by \cite{mw}: 
\begin{equation*}
\left[
\begin{matrix}
-1&1&1\\1&-1&1\\1&1&-1
\end{matrix}
\right],\;
\left[
\begin{matrix}
0&s_2^{-1}&s_3^{-1}\\s_1^{-1}&0&s_3^{-1}\\s_1^{-1}&s_2^{-1}&0
\end{matrix}
\right],\;
\left[
\begin{matrix}
0&\frac{s_1 s_3}{s_1-s_2+s_3}&\frac{s_1 s_2}{s_1+s_2-s_3}\\
\frac{s_2 s_3}{-s_1+s_2+s_3}&0&\frac{s_1 s_2}{s_1+s_2-s_3}\\
\frac{s_2 s_3}{-s_1+s_2+s_3}&\frac{s_1 s_3}{s_1-s_2+s_3}&0
\end{matrix}
\right]
\end{equation*}

A few derived triangles are shown in Figure~\ref{fig:derived-isosceles}, showing a property similar to Lemma~\ref{lem:axis-of-symmetry}, Appendix~\ref{app:method-lemmas}, namely, when the 3-periodic is an isosceles, one vertex of the derived triangle lies on the orbit's axis of symmetry.

\begin{figure}[H]
    \centering
    \includegraphics[width=\textwidth]{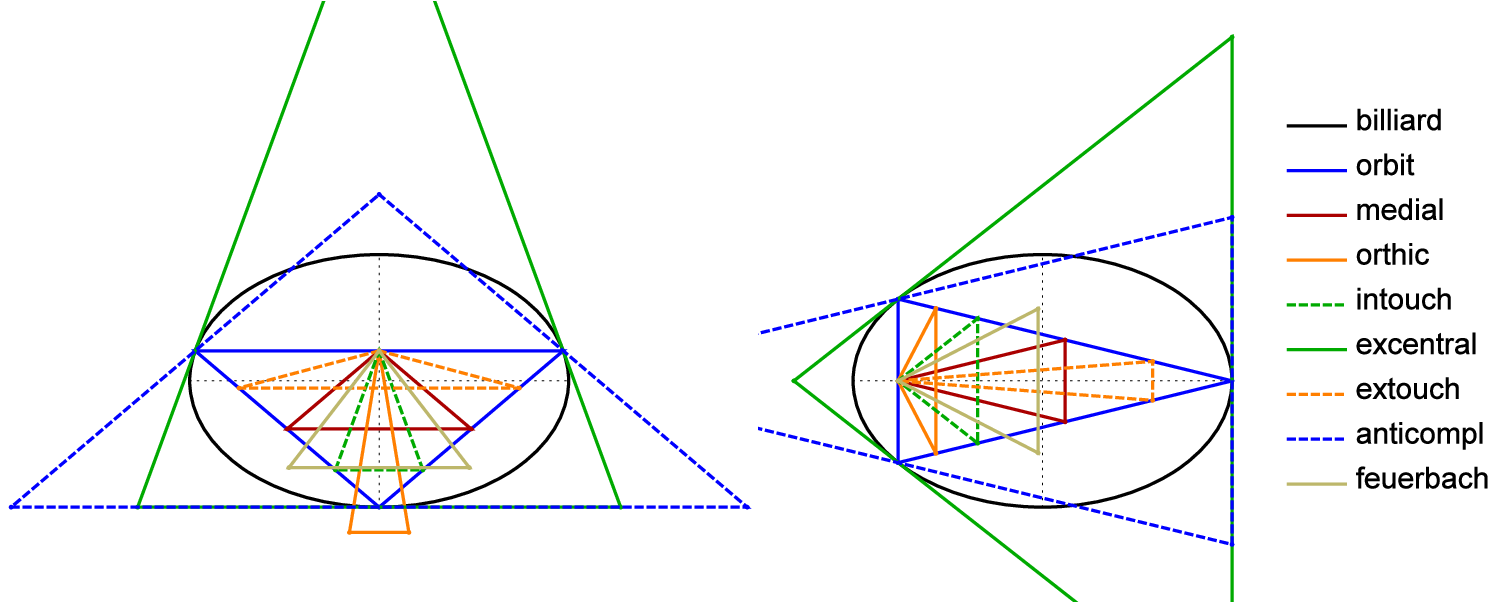}
    \caption{When the orbit is an isosceles triangle (solid blue), any derived triangle will contain one vertex on the axis of symmetry of the orbit. \href{https://youtu.be/xyroRTEVNDc}{Video}}
    \label{fig:derived-isosceles}
\end{figure}

\section{Review: Elliptic Billiards}
\label{app:billiards}
An Elliptic Billiard (EB) is a particle moving with constant velocity in the interior of an ellipse, undergoing elastic collisions against its boundary \cite{rozikov2018,sergei91}, Figure~\ref{fig:billiard-trajectories}. For any boundary location, a given exit angle (e.g., measured from the normal) may give rise to either a quasi-periodic (never closes) or $N$-periodic trajectory \cite{sergei91}, where $N$ is the number of bounces before the particle returns to its starting location. All trajectory segments are tangent to a confocal caustic \cite{sergei91}. The EB is a special case of {\em Poncelet's Porism} \cite{dragovic11}: if one $N$-periodic trajectory can be found departing from some boundary point, any other such point will initiate an $N$-periodic, i.e., a 1d {\em family} of such orbits will exist. A classic result is that $N$-periodics conserve perimeter \cite{sergei91}.

\begin{figure}
    \centering
    \includegraphics[width=\textwidth]{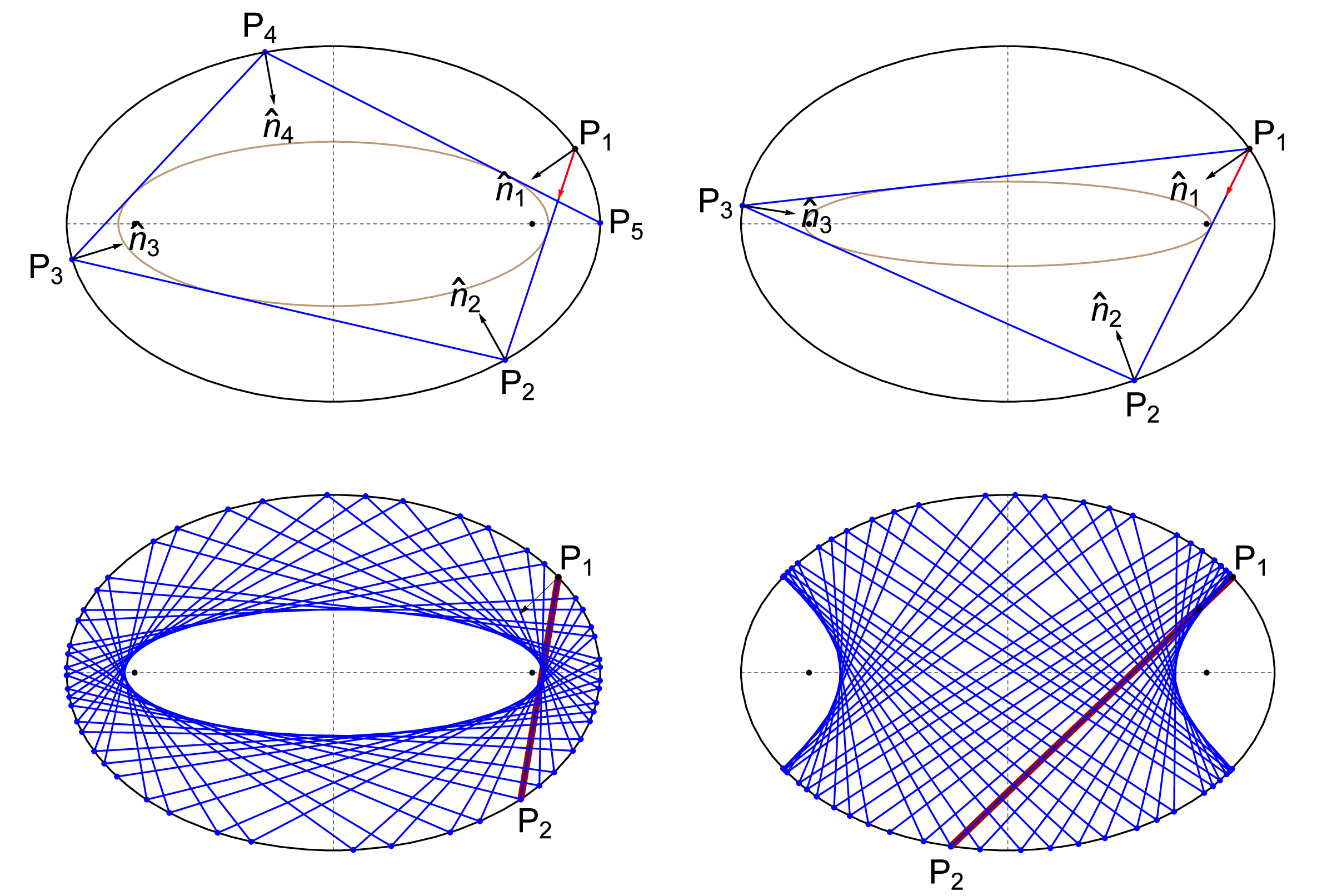}
    \caption{Trajectory regimes in an Elliptic Billiard, taken from \cite{reznik2020-intelligencer}. \textbf{Top left}: first four segments of a trajectory departing at $P_1$ and toward $P_2$, bouncing at $P_i, i=2,3,4$. At each bounce the normal $\hat{n}_i$ bisects incoming and outgoing segments. Joachimsthal's integral \cite{sergei91} means all segments are tangent to a confocal {\em caustic} (brown). \textbf{Top right}: a 3-{\em periodic} trajectory. All 3-periodics in this Billiard will be tangent to a special confocal caustic (brown). \textbf{Bottom}: first 50 segments of a non-periodic trajectory starting at $P_1$ and directed toward $P_2$. Segments are tangent to a confocal ellipse (left) or hyperbola (right). The former (resp. latter) occurs if $P_1P_2$ passes outside (resp. between) the EB's foci (black dots).}
    \label{fig:billiard-trajectories}
\end{figure}

\section{Expressions used in Section~\ref{sec:rational-trilinears}}
\label{app:rational-support}
Let the boundary of the Billiard satisfy Equation~\eqref{eqn:billiard-f}. Assume, without loss of generality, that $a{\geq}b$. Here we provide expressions used in Section~\ref{sec:algebraic}. Let $P_1,P_2,P_3$ be an orbit's vertices.

\subsection{Exit Angle Required for 3-Periodicity}
\label{app:exit-angle}

Consider a starting point $P_1=(x_1,y_1)$ on a Billiard with semi-axes $a,b$. The cosine of the exit angle $\alpha$ (measured with respect to the normal at $P_1$, i.e., $\alpha=\theta_1/2$) required for the trajectory to close after 3 bounces is given by \cite{garcia2019-incenter}: 
\begin{equation*}
\cos^2{\alpha}=\frac{d_1^2\delta_1^2}{\,d_2}=k_1,\;\;\;\sin{\alpha}\cos\alpha=\frac{ \delta_1d_1^2}{d_2 }\sqrt{ d_2 -d_1^4\delta_1^2}=k_2
\end{equation*}
\noindent with 
$c^2=a^2-b^2$, $d_1=(a\,b/c)^2$, $d_2={b}^{4}x_1^2 +{a}^{4}y_1^2$, $\delta=\sqrt{a^4+b^4-a^2 b^2}$, and $\delta_1=\sqrt{2 \delta-a^2-b^2}$.

\subsection{Orbit Vertices}
\label{app:p1p2p3}
Given a starting vertex $P_1=(x_1,y_1)$ on the EB, $P_2=(p_{2x},p_{2y})/q_2$, and $P_3=(p_{3x},p_{3y})/q_3$ where \cite{garcia2019-incenter}:
\begin{align*}
\small
p_{2x}=&-{b}^{4} \left(  \left(   a^2+{b}^{2}\right)k_1 -{a}^{2}  \right) x_1^{3}-2\,{a}^{4}{b}^{2} k_2  x_1^{2}{y_1}\\
&+{a}^{4} \left(  ({a
}^{2}-3\, {b}^{2})k_1+{b}^{2}
 \right) {x_1}\,y_1^{2}-2{a}^{6} k_2 y_1^{3}\\
p_{2y}=& 2{b}^{6} k_2 x_1^{3}+{b}^{4}\left(  ({b
 }^{2}-3\, {a}^{2}) k_1  +{a}^{2}
  \right) x_1^{2}{y_1}\\
&+  2\,{a}^{2} {b}^{4}k_2 {x_1} y_1^{2} -{
a}^{4}  \left(  \left(   a^2+{b}^{2}\right)k_1  -{b}^{2}  \right)  y_1^{3}
\\
q_2=&{b}^{4} \left( a^2-c^2k_1   \right)
x_1^{2}+{a}^{4} \left(  {b}^{2}+c^2k_1  
 \right) y_1^{2} - 2\, {a}^{2}{b}^{2}{c^2}k_2 {x_1}\,{
y_1} \\
p_{3x}=& {b}^{4} \left( {a}^{2}- \left( {b}^{2}+{a}^{2} \right) \right)
 k_1  x_1^{3} +2\,{a}^{4}{b}^{2}k_2  x_1^{2}{ y_1}\\
 &+{a}^{4} \left( 
  k_1 \left( {a}^{2}-3\,{b}^{2}
 \right) +{b}^{2} \right) { x_1}\, y_1^{2} +2\, {a}^{6} k_2 y_1^{3}
\\
p_{3y}=& -2\, {b}^{6} k_2 x_1^{3}+{b}^{4} \left( {a}^{2}+ \left( {b}^{2}-3\,{a}^{2} \right)    k_1 \right) {{ x_1}}^{2}{ y_1}
\\
& -2\,{a}^{2}  {b}^{4} k_2  x_1 y_1^{2}+
 {a}^{4} \left( {b}^{2}- \left( {b}^{2}+{a}^{2} \right)   k_1 \right)\,  y_1^{3},
\\
q_3=& {b}^{4} \left( {a}^{2}-{c^2}k_1   \right) x_1^{2}+{a}^{4} \left( {b}^{2}+c^2k_1  \right)  y_1^{2}+2\,{a}^{2}{b}^{
2} c^2 k_2\, { x_1}\,{ y_1}.
\end{align*}

\subsection{Polynomials Satisfied by the Sidelengths}
\label{app:alg_locus}
 
Let sidelengths $s_1=|P_3-P_2|, s_2=|P_1-P_3|,s_3=|P_2-P_1|$. The following are polynomial expressions on $s_i$ and $u,u_1,u_2$:
{\small  
\begin{align*}
		g_1=&    -h_1 \, s_1^2 
		+h_0,\;\;g_2=- h_1\, s_2^2 -
		h_2 \,u_1 \,u_2 +h_3,\;\;\;g_3=- h_1 \,s_3^2 + h_2\,  u_1\, u_2+h_3\\
		h_0=& 12c^{12}(a^2+b^2+2\delta)u^8+24c^{10}(a^2b^2+2b^4-2\delta c^2 u^6\\
		-&4c^4[10a^{10}-12a^8b^2+11a^6b^4-7a^4b^6+24a^2b^8-8b^{10}\\
		-&  (2(4a^4-2a^2b^2+b^4))(a^4-2a^2b^2+4b^4)
		\delta
		]u^4\\
		+&8a^6 c^2[4a^6-7a^4b^2+11a^2b^4-2b^6-(2(a^2+ab+b^2))(a^2-ab+b^2)\delta]u^2\\
		+&4a^{12} \delta_1^2	\\
		h_1=&  c^2\left(3c^4 u^4-2c^2(a^2-2b^2 u^2-a^4)\right)^2\\
		h_2=&-2{a(b^2-\delta)}\delta_1 u ((3 c^6 \delta+(6 (a^2+b^2)) c^6) u^6+((3 (a^2+4 b^2)) c^4 \delta \\
		-& (3 (2 a^4-a^2 b^2-4 b^4)) c^4) u^4
		+( (b_2-a^2 ) (7 a^4-8 b^4) \delta\\
		+& 2 c^2  (a^2-2 b^2) (a^4-2 b^4)) u^2+a^4 (a^2-4 b^2) \delta-a^4 (2 a^4-3 a^2 b^2+4 b^4))\\
		h_3=&c^6 (9 c^{6} u^{10}+ 1)((18 (a^4+b^4))  \delta-(3 (7 a^6-16 a^4 b^2+29 a^2 b^4-8 b^6)) u^8\\
		+&2c^2\left[     (13 a^8-12 a^6 b^2+49 a^4 b^4-54 a^2 b^6+16 b^8)\right.\\
		-& \left. 2   (10 a^6-7 a^4 b^2+11 a^2 b^4-8 b^6) \delta 
		\right] u^6\\
		+&((28a^4\delta^2-40 a^2 b^6+16 b^8) \delta 
		-  26 a^{10}+12 a^8 b^2+42 a^4 b^6-48 a^2 b^8+16 b^{10}) u^4\\
		+&a(13 a^6-9 a^4 b^2-8 b^4c^2 -(4 (2 a^4+a^2 b^2-2 b^4)) \delta)^4 u^2+2 a^8 ( \delta-a^8
		 c^2)
\end{align*}
}%
	




\section{Proof of Lemmas used in Section~\ref{sec:loci_geom}}
\label{app:method-lemmas}
\begin{lemma*}[\ref{lem:center-cover}]
A parametric traversal of $P_1$ around the EB boundary is a triple cover of the locus of a triangle center $X_i$.
\end{lemma*}

\begin{proof}
Let $P_1(t)=\left(a \cos{t},b \sin{t}\right)$. Let $t^*$ (resp.~$t^{**})$ be the value of $t$ for which the 3-periodic orbit is an isosceles with a horizontal (resp. vertical) axis of symmetry, Figure~\ref{fig:triple-winding}, $t^*>t^{**}$. For such cases one can easily derive\footnote{Indeed, $\sin{t^*}=J b$ and $\cos{t^{**}}=J a$, where $J$ is Joachmisthal's constant ($N=3$) \cite{sergei91}.}:

\begin{equation*}
\tan{t^*}=\frac{b\sqrt{2\delta-{a}^{2}+2\,{b}^{2}}}{a^{2}}, \;\;\;\tan{t^{**}}=  \frac{ \sqrt {2\,\delta-2\,{a}^{2}+{b}^{2}}}{\sqrt{3}\, a}
\end{equation*}

\noindent with $\delta=\sqrt{a^4-a^2 b^2+b^4}$ as above. Referring to Figure~\ref{fig:triple-winding}:

\begin{affirmation}
\label{obs:tstar1}
A continuous counterclockwise motion of $P_1(t)$ along the intervals $[-t^*,-t^{**})$, $[-t^{**},0)$, $[0,t^{**})$, and $[t^{**},t^*)$ will each cause $X_i$ to execute a quarter turn along its locus, i.e., with $t$ varying from $-t^*$ to $t^*$, $X_i$ will execute one complete revolution on its locus\footnote{The direction of this revolution depends on $X_i$ and its not always monotonic. The observation is valid for elliptic or non-elliptic loci alike. See \cite{helman2021-theory} for details.}.
\end{affirmation}

\begin{affirmation}
\label{obs:tstar2}
A continuous counterclockwise motion of $P_1(t)$ in the $[t^*,\pi-t^{*})$ (resp.~$[\pi-t^{*},\pi+t^{*})$, and $[\pi+t^{*},2\pi-t^*)$), visits the same 3-periodics as when $t$ sweeps $[-t^*,-t^{**})$ (resp.~$[-t^{*},t^{*})$, and $[t^*,\pi-t^*)$).
\end{affirmation}

Therefore, a complete turn of $P_1(t)$ around the EB visits the 3-periodic family thrice, i.e., $X_i$ will wind thrice over its locus.
\end{proof}

\begin{figure}
    \centering
    \includegraphics[width=.9\textwidth]{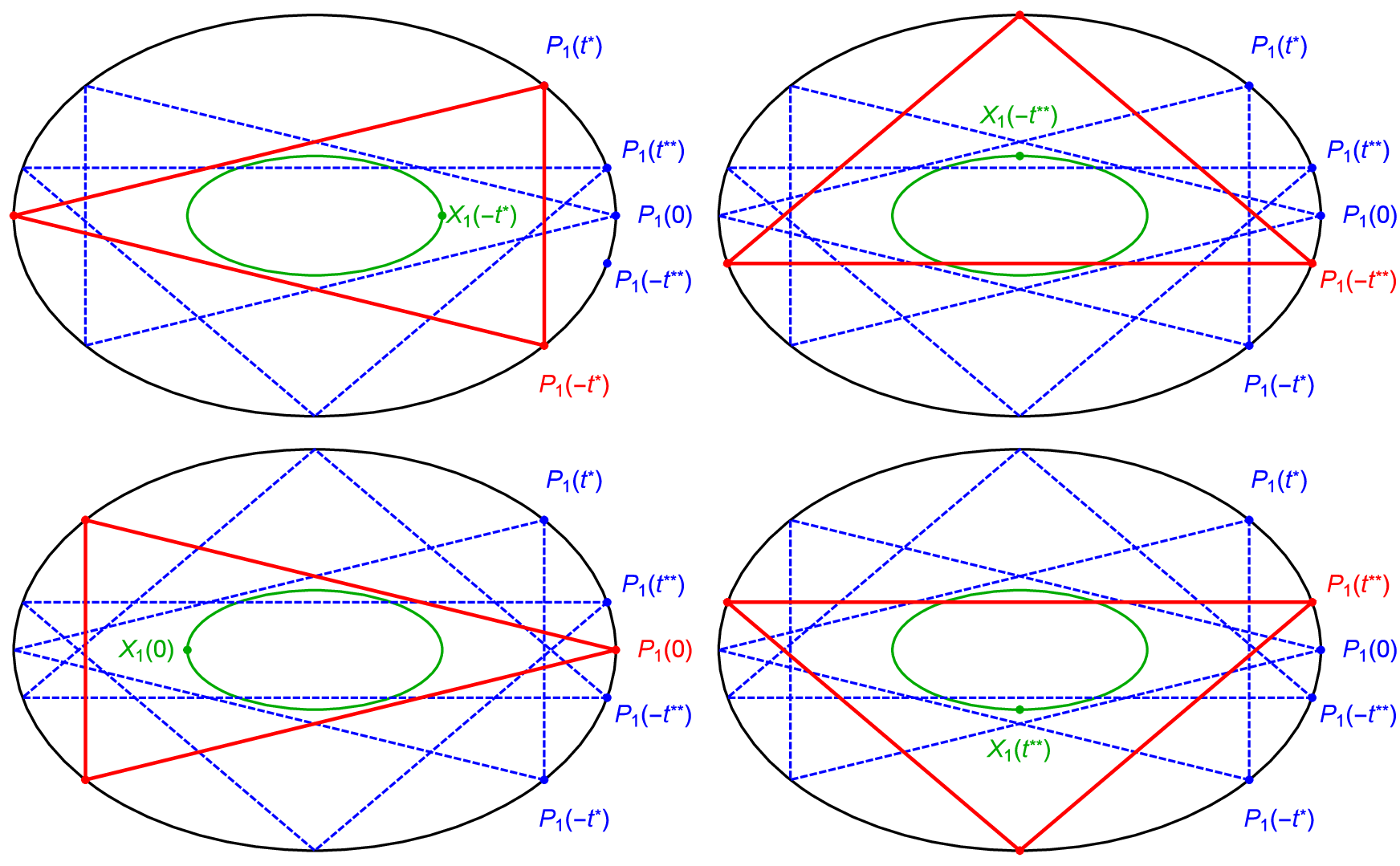}
    \caption{Counterclockwise motion of $P_1(t)$, for $t=[-t^*,t^*)$ can be divided in four segments delimited by $t=(-t^*,-t^{**},0,t^{**},t^{**})$. Orbit positions for the first four are shown (red polygons) in the top-left, top-right, bot-left, and bot-right pictures. At $P_1({\pm}t^*)$ (resp.~$P_1({\pm}t^{**})$) the orbit is an isosceles triangle with a horizontal (resp.~vertical) axis of symmetry. Observations~\ref{obs:tstar1},\ref{obs:tstar2} assert that a counterclockwise motion of $P_1(t)$ along each of the four intervals causes a triangle center $X_i$ to execute a quarter turn along its locus (elliptic or not), and that a complete revolution of $P_1(t)$ around the EB causes $X_i$ to wind thrice on its locus. For illustration, the locus of $X_1(t)$ is shown (green) at each of the four interval endpoints.}
    \label{fig:triple-winding}
\end{figure}
Below we provide proofs of Lemmas \ref{lem:axis-of-symmetry}, \ref{lem:axisymmetric} and \ref{lem:center-cover}.
\begin{lemma*}[\ref{lem:axis-of-symmetry}]
Any triangle center $X_i$ of an isosceles triangle is on the axis of symmetry of said triangle.
\end{lemma*}

\begin{proof} Consider a sideways isosceles triangle with vertices $P_1=(x_1,0)$, $P_2=(-x_2,y_2)$
and $P_3= (-x_2,-y_2)$, i.e., its axis of symmetry is the $x$ axis. Let $X_i$ have trilinears $p:q:r=h(s_1,s_2,s_3):h(s_2,s_3,s_1):h(s_3,s_1,s_2)$. Its cartesian coordinates are given by Equation~\eqref{eqn:trilin-cartesian}. As $s_2=s_3$ and $h$ is symmetric on its last two variables $h(s_1,s_2,s_3)=h(s_1,s_3,s_2)$ it follows from equation \eqref{eqn:trilin-cartesian} that $y_i=0$. 
\end{proof}

\begin{lemma*}[\ref{lem:axisymmetric}]
If the locus of triangle center $X_i$ is elliptic, said ellipse must be concentric and axis-aligned with the EB.
\end{lemma*}

\begin{proof}
This follows from Lemma~\ref{lem:axisymmetric}.
\end{proof}

\begin{lemma*}[\ref{lem:center-cover}]
If the locus of $X_i$ is an ellipse, when $P_1(t)$ is at either EB vertex, its non-zero coordinate is equal to the corresponding locus semi-axis length.
\end{lemma*}

\begin{proof}
The family of 3-periodic orbits contains four isosceles triangles, Figure~\ref{fig:sideways-upright-orbit}. Parametrize $P_1(t)=(a\cos t,b\sin t)$. It follows from Lemma \ref{lem:axisymmetric} that $X_i(0)=(\pm a_i,0)$, $X_i(\pi/2)=(0,\pm b_i),$  $X_i(\pi)=(\mp a_i, 0)$ and $X_i(3\pi/2)=(0,\mp b_i)$, for some $a_i,b_i$. This ends the proof. 
\end{proof}



\section{Review: Early Observations}
\label{app:early}
Figure~\ref{fig:x12345-feuer-combo} (left) shows the elliptic loci of $X_k$, $k=1,...,5$. The latter two are new results.

\begin{figure}
\centering
\includegraphics[width=\linewidth]{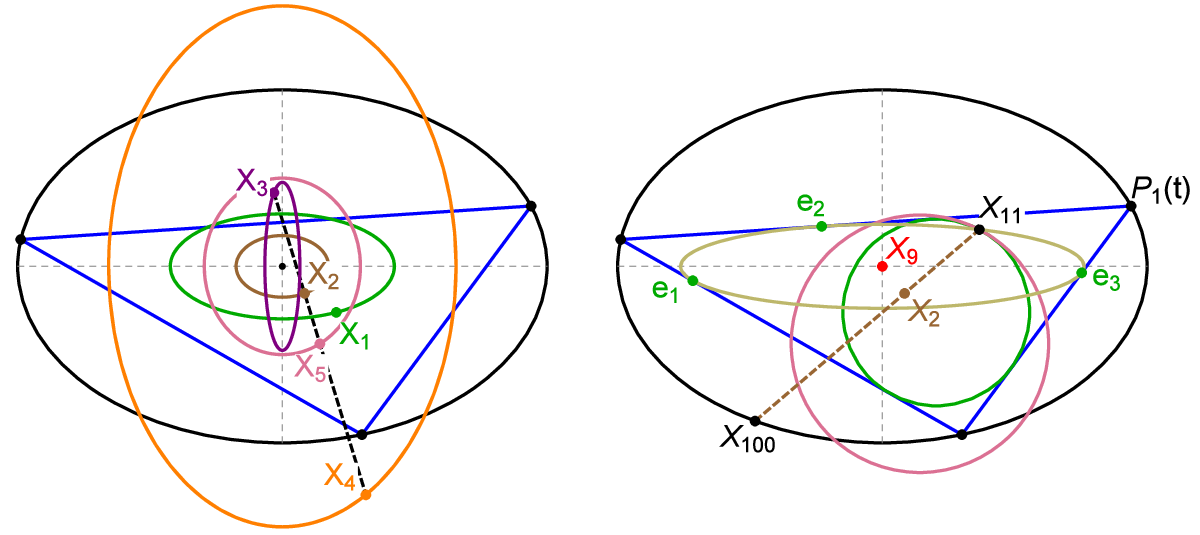}
\caption{\textbf{Left}: The loci of Incenter $X_1$, Barycenter $X_2$, Circumcenter $X_3$, Orthocenter $X_4$, and center of the 9-Point Circle $X_5$ are all ellipses, \href{https://youtu.be/sMcNzcYaqtg}{Video}, \href{https://bit.ly/3opClJq}{app}. Also shown is the {\em Euler Line} (dashed black) which for any triangle, passes through all of $X_i$, $i=1...5$ \cite{mw}. \textbf{Right}: A 3-periodic orbit starting at $P_1(t)$ is shown (blue). The locus of $X_{11}$, where the Incircle (green) and 9-Point Circle (pink) meet, is the caustic (brown), also swept by the Extouchpoints $e_i$. $X_{100}$ (double-length reflection of $X_{11}$ about $X_2$) is the EB \href{https://youtu.be/TXdg7tUl8lc}{Video}, \href{https://bit.ly/2LuARPo}{app}.}
\label{fig:x12345-feuer-combo}
\end{figure}

We have also observed that the locus of the Feuerbach Point $X_{11}$ coincides with the $N=3$ {\em caustic}, a confocal ellipse to which the 3-periodic family is internally tangent, Appendix~\ref{app:billiards}. Additionally, the locus of $X_{100}$, the anticomplement\footnote{Double-length reflection about the Barycenter $X_2$.} of $X_{11}$ is the EB boundary. These phenomena appear in Figure~\ref{fig:x12345-feuer-combo} (right).

Indeed, some these facts will be known to triangle specialists if one regards the EB as the circumellipse centered on the Mittenpunkt $X_9$ \cite[X(9)]{etc} and \cite{dekov14}. For example, a full 57 triangle centers lie on said circumellipse. An animation of some of these points is viewable in \cite[pl\#10]{dsr_playlist_2020}.

\begin{figure}
    \centering
    \includegraphics[width=.85\textwidth]{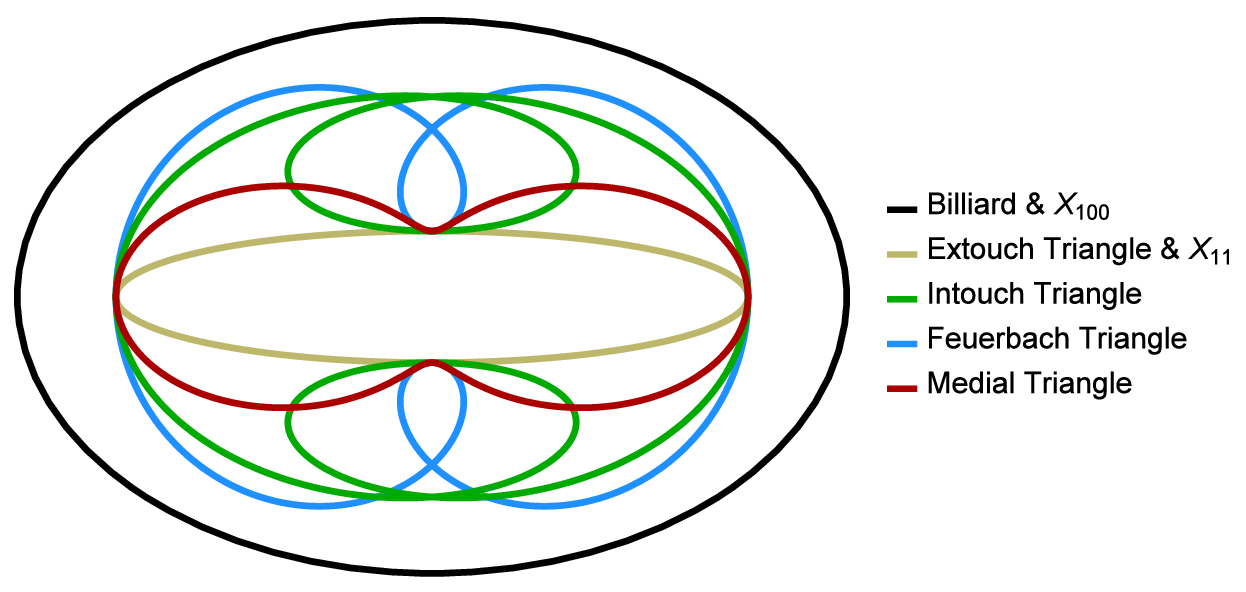}
    \caption{Loci generated by the vertices of selected orbit-derived triangles, namely: the Intouch (green), Feuerbach (blue, not to be confused with the Feuerbach {\em Point} $X_{11}$), and Medial (red) Triangles are non-elliptic. However, those of the Extouch Triangle (brown), are identical to the $N=3$ caustic (a curve also swept by $X_{11}$). Not shown is the locus of the Excentral Triangle, an ellipse similar to a rotated copy of the Incenter locus. \href{https://bit.ly/2MMH9e5}{app},
    \href{https://youtu.be/9xU6T7hQMzs}{Video 1}, \href{https://youtu.be/Xxr1DUo19_w}{Video 2}, \href{https://youtu.be/TXdg7tUl8lc}{Video 3}, \href{https://youtu.be/OGvCQbYqJyI}{Video 4}.}
    \label{fig:non-elliptic-vertex}
\end{figure}

Additionally, a few  observations have been made \cite{reznik2020-intelligencer} about the loci of vertices of orbit-derived triangles (see Appendix~\ref{app:derived-tris}), some of which are elliptic and others non, illustrated in Figure~\ref{fig:non-elliptic-vertex}.

\bibliographystyle{maa} 
\bibliography{authors_rgk,references}

\end{document}